%% file: CMC-construction-11-2-2015-ARXIV.tex
\documentclass[11pt,oneside,letter]{article}
\input{headingsIgor}
\usepackage{cancel}
\usepackage{ulem}
\usepackage[T1]{fontenc}
\usepackage[utf8]{inputenc}

\usepackage{upgreek}
\title{  Conditional Markov Chains Revisited \\
Part I:  Construction and properties. }

\author{Tomasz R.  Bielecki$^1$ \\[-0.3ex]
\url{bielecki@iit.edu} \\[-0.3ex]
\and
Jacek Jakubowski$^{2,3}$ \\[-0.3ex]
\url{jakub@mimuw.edu.pl} \\[-0.3ex]
\and
Mariusz Niew\k{e}g\l owski$^3$
\\[-0.3ex]
\url{m.nieweglowski@mini.pw.edu.pl} \\[-0.3ex]
\and \small{$^1$Department of Applied Mathematics,}\\[-0.3ex]
\small{Illinois Institute of Technology,}\\[-0.3ex]
\small{Chicago, IL 60616, USA }\\[-0.3ex]
\and
\small{$^2$Institute of Mathematics,}\\[-0.3ex]
\small{University of Warsaw,}\\[-0.3ex]
\small{Banacha 2,
02-097 Warszawa, Poland}\\[-0.3ex]
\and
\\[-0.3ex]
\small{$^3$Faculty of Mathematics
and Information Science,}\\[-0.3ex]
\small{Warsaw University of Technology,}\\[-0.3ex]
\small{ul. Koszykowa 75,
00-662 Warszawa, Poland}\\[-0.3ex]
}
\date{\today}

\binoppenalty=\maxdimen 
\relpenalty=\maxdimen

\usepackage{dsfont}
\newcommand{\be}{\begin{equation}}
\newcommand{\ee}{\end{equation}}
\newcommand{\bde}{\begin{displaymath}}
\newcommand{\ede}{\end{displaymath}}
\newcommand{\beq}{\begin{eqnarray*}}
\newcommand{\eeq}{\end{eqnarray*}}
\newcommand{\beqa}{\begin{eqnarray}}
\newcommand{\eeqa}{\end{eqnarray}}
\newcommand{\bel }{\left\{\begin{array}{ll}}
\newcommand{\eel}{\cr \end{array} \right.}
\newcommand{\bd}{\begin{definition} \rm }
\newcommand{\ed}{\end{definition} \rm }
 \newcommand{\bex}{\begin{example} \rm }
\newcommand{\eex}{\end{example}}
\newcommand{\bt}{\begin{theorem}}
\newcommand{\et}{\end{theorem}}
\newcommand{\bl}{\begin{lemma}}
\newcommand{\el}{\end{lemma}}
\newcommand{\bp}{\begin{proposition}}
\newcommand{\ep}{\end{proposition}}
\newcommand{\bcor}{\begin{corollary}}
\newcommand{\ecor}{\end{corollary}}
\newcommand{\lab }{\label }
\newcommand{\brem}{\begin{remark}}
\newcommand{\erem}{\end{remark}}

\def\finproof{\hfill $\Box$ \vskip 5 pt}

\def\cadlag{c\` adl\` ag\ }
\def\I{\mathds{1}}

\def \wh{\widehat}
\def \wt{\widetilde}
\def\r{\mathbb R}
\def\F{{\mathcal F}}
\def\H{{\mathcal H}}
\def\G{{\mathcal G}}

\def\FF{{\mathbb F}}
\def\HH{{\mathbb H}}
\def\GG{{\mathbb G}}

\def\P{\mathbb P}
 \def\Q{\mathbb Q}

\def\E {{\mathbb E} }
\def\tT{{t\in[0,T]}}

\def\wrt{\mbox{with respect to }}

\DeclareMathAlphabet\mathbfcal{OMS}{cmsy}{b}{n}

\newcommand{\mn}[1]{\begin{color}[rgb]{0.20, 0.66, 0.26}#1\end{color}}

\def\j{\textcolor[rgb]{0.85, 0.35, 0.00}}

 \mathchardef\mhyphen="2D
    \usepackage{fancyhdr}
    \usepackage{lastpage}
\date{ \today}
\begin{document}
\maketitle
\begin{abstract}
In this paper we continue the study of conditional Markov chains (CMCs) with finite state spaces, that we initiated in Bielecki, Jakubowski and Niew\k{e}g\l owski (2015) in an effort to enrich the theory of CMCs that was originated in Bielecki and Rutkowski (2004). We provide an alternative definition of a CMC and an alternative construction of a CMC via a change of probability measure. It turns out that our construction produces CMCs that are also doubly stochastic Markov chains (DSMCs), which allows to study of several properties of CMCs using tools available for DSMCs.\\
{\noindent \small
  }
\end{abstract}
\tableofcontents

\section{Introduction}

In this paper we continue the study of conditional Markov chains (CMCs) with finite state spaces, that we initiated in Bielecki, Jakubowski and Niew\k{e}g\l owski \cite{BieJakNie2015} in an effort to enrich the theory of CMCs that was originated in Bielecki and Rutkowski \cite{BieRut2004}.

CMCs were conceptualized in the context of credit risk, where they have been found to provide a useful tool for modeling credit migrations. In many ways, a CMC is an important generalization of the concept of a default time with stochastic compensator, {a key concept in the models of financial markets allowing for default of  parties of a financial contract}. Such a model of default time is really just a special example of a CMC: it is a CMC taking values in a state space consisting of only two states, say $0$ and $1$, where $0$ is the transient state and $1$ is the absorbing state.

In \cite{BieJakNie2015} we proposed a modified definition of the conditional Markov property, which was less general than Definition 11.3.1 used in Chapter 11.3 in \cite{BieRut2004}. The reason for this was that the definition of  conditional Markov property proposed in \cite{BieJakNie2015} was aimed at providing a suitable   framework  for study of Markov consistency properties for conditional Markov chains and study of Markov copulae for conditional Markov chains, a feature that can not be achieved within the framework of the CMC framework proposed in  \cite{BieRut2004}. Still, the definition of the conditional Markov property, and the related construction of a CMC as presented in \cite{BieJakNie2015} were not general enough, as they did not allow for study of conditional Markov families. This is because in \cite{BieJakNie2015} we only dealt with processes starting from a fixed, non-random, initial state.
Here, we generalize the definition of a conditional Markov property and construction of a CMC that allow for the initial state of the chain to have a nondegenerate conditional initial distribution, and, consequently, allow for study of conditional Markov families. Such study will be conducted elsewhere.

Classical conditional Markov chains, that is, the ones defined originally in \cite{BieRut2004}, have already proven to play important role in applications in finance and in insurance, for example (cf.
Bielecki and Rutkowski \cite{BieRut2000},
\cite{BieRut2003},
\cite{BieRut2004},
Jakubowski and Niew\k{e}g{\l}owski
\cite{JakNie2011},
Eberlein and {\"O}zkan
\cite{EbeOzk2003},
Eberlein and Grbac
\cite{EbeGrb2013},
Biagini, Groll and Widenmann
\cite{BiaWid2013}).  The main advantage of these processes is that, via appropriate conditioning, their primary Markov properties are mixed with dependence of their infinitesimal characteristics on relevant random factors, that do not have to be Markovian. {The CMCs studied in this paper may lead to many more applications since, as already has been mentioned above, the present modified definition allows to study dependence properties between CMCs, which are crucial in applications to credit and counterparty risk, among other applications.}
In fact, the present paper is a companion paper to Bielecki, Jakubowski and Niew\k{e}g\l owski \cite{BieJakNie2014b}, where we complement the study done here by investigating the issues of modeling dependence between CMCs, and we propose some specific applications.

An important family of jump processes, so called doubly stochastic Markov chains (DSMC), was introduced in Jakubowski and Niew\k{e}g\l owski \cite{JakNie2010}. The conditional Markov chains constructed in the present paper turn out to be doubly stochastic Markov chains. Thus, the benefit from the construction provided here is two-fold:
\begin{itemize}
\item The constructed CMCs enjoy the conditional Markov property, which has unquestionable practical appeal, and
\item The constructed CMCs enjoy the doubly stochastic Markov property, which has critical theoretical implications allowing for applying important tools from stochastic analysis to studying CMCs.
\end{itemize}

The paper is organized as follows: In Section \ref{sec:CMC-and-int} we introduce the concept of CMC, which underlies the present study. In this section we also introduce and discuss the relevant concept of stochastic generator (or an intensity matrix) of a CMC. In addition, we give there two examples of $(\FF,\GG)$-CMC, one  which does not have the intensity, and one with the intensity. Section \ref{constr} is devoted to presentation of a specific method for constructing a CMC. In Section \ref{cmcdsmc} we relate conditional Markov chains to doubly stochastic Markov chains. In particular, we show that any conditional Markov chain constructed using the change of measure technique used in Section \ref{constr} is also a doubly stochastic Markov chain. Finally, in the last section we collect all needed technical results used throughout the paper.

\section{Conditional Markov Chain and Its Intensity}
\label{sec:CMC-and-int}

Let $T>0$ be a fixed finite time horizon. Let $(\Omega, \mathcal{A}, \mathbb{P})$ be an underlying complete probability space, which is endowed with two
filtrations, $\FF=(\F_t)_{t\in[ 0,T]}$ and $\GG=(\G_t)_{t\in[ 0,T]}$, that are assumed to satisfy the usual conditions.
{The standing assumption though will be that all filtrations used in the paper are completed, with respect to relevant probability measures.}
For the future reference we also define
\be\label{eq:GGhat-0}
\wh{\G_t} := \F_T \vee \G_t, \quad \tT,
\ee
as well as the corresponding filtration $\wh{\GG} := (\wh{\G}_t)_{ t \in [0,T]}.$ 
{
In what what follows, we will not require that it is right-continuous.
}
%

Typically, processes considered in this paper are defined on  $(\Omega, \mathcal{A}, \mathbb{P})$, and are restricted to the time interval $[0,T]$. { Moreover, for any process $U$
we denote  by $\FF^U$ the completed right-continuous filtration generated by this process.} In addition, we fix a finite set $S$, and we denote by $d$ the cardinality of $S$. Without loss of generality we take $S=\set{1,2,3,\ldots,d}.$

\begin{definition}\label{def:CMC}
An $S$-valued,  $\mathbb{G}$-adapted c\` adl\` ag  process $X$ is called  an $(\mathbb{F},\mathbb{G})$--conditional Markov chain if for every $x_1, \ldots, x_k \in S$ and  for every $0\leq t\leq t_1 \leq \ldots \leq t_k\leq T$ it satisfies the following property$\, $
\begin{equation}\label{eq:CMC-def-1}
\P( X_{t_1} =x_1, \ldots, X_{t_k} = x_k | \F_t \vee \G_t)
=
\P( X_{t_1} =x_1, \ldots, X_{t_k} = x_k | \F_t \vee \sigma(X_t)).
\end{equation}
\end{definition}

\begin{remark} (i) We will call filtration $\GG$ the \textsl{base} filtration, and we will call filtration $\FF$ the \textsl{reference} filtration.
Usually $\GG = \FF^X$.
\\
(ii)  {It needs to be stressed that an  $(\mathbb{F},\mathbb{G})$--conditional Markov chain may not be a classical Markov chain (in any filtration).} However, if  $\mathbb{G}$ is independent of $\mathbb{F}$, then the above definition reduces to the case of a classical Markov chain with respect to filtration $\GG$, or $\GG\,$--Markov chain.
In other words, a classical $\GG$-Markov chain is an  $(\mathbb{F},\mathbb{G})$--conditional Markov chain for the reference filtration independent of  the base filtration.
\end{remark}

In what follows we shall write  $(\FF, \GG)$-CMC, for short, in place of  $(\FF, \GG)$-conditional Markov chain.

\subsection{Intensity of an $(\FF,\GG)$-CMC}

Let $X$ be an $(\FF,\GG)$-CMC. For each $x \in S$ we define the corresponding state indicator process of $X$,
\begin{equation}\label{eq:def-Hi}
H^x_t := \I_\set{ X_t = x }, \quad t\in[0,T].
\end{equation}
Accordingly, we define a column vector  $ H_t =(H^x_t, x\in S)^\top $, where $\top$
denotes transposition. Similarly, for $x,y\in S,\ x\ne y,$ we define process $H^{xy}$  that counts the number of transitions from $x$ to $y$,
\begin{equation}\label{eq:def-Hxy}
	H^{xy}_t := \# \set{ u
\leq t: X_{u-} = x\  \textrm{and}\ X_u =y } = \int_{] 0,t]} \!\!\!H^{x}_{u-} d H^{y}_u, \quad \tT.
\ee

The following definition generalizes the concept of the generator matrix (or intensity matrix) of a Markov chain.

\begin{definition}\label{def:F-intensity}
We say that an $\FF$-adapted (matrix valued) process $\Lambda_t=[\lambda^{xy}_t]_{x,y \in S }$ such that
\begin{align}\label{eq:int-cond}
	\lambda^{xy}_t \geq 0, \quad \forall x, y \in S, x \neq y, \quad  \text{and} \quad
	\sum_{y \in S }\lambda^{xy}_t = 0, \quad \forall x \in S,
\end{align}
is an $\FF$-stochastic generator or an $\FF$-intensity matrix process for $X$, if the process $M := (M^x, x\in S)^\top$ defined as
\be\label{eq:Mart-M-nowy}
     M_t = H_t - \int_0^t \Lambda_u^\top H_{u} du, \quad t\in[0,T],
\ee
is an $\FF \vee \GG\,$--$\,$local martingale (with values in $\r^d$).
\end{definition}
\brem\label{rem:cacy}
We remark that even though the above definition is stated for an $(\FF,\GG)$-CMC process $X$, it applies to $S$-valued semimartingales.
\erem

We will now discuss the question of uniqueness of $\FF$-intensity.

\begin{definition}
We say that two processes  $\Lambda$ and $\wh \Lambda$ are equivalent relative to $X$ if
\be\label{eq:rown}
\int_0^t (\Lambda_u - \wh \Lambda_u)^\top H_{u} du = 0, \quad \forall t \in [0,T].
\ee
\end{definition}

\begin{proposition}\label{rem:int-wersje}
Let $X$ be an $(\FF,\GG)$-CMC.
\begin{itemize}
\item[i)] If $\Lambda$ and $\wh \Lambda$ are $\FF$-intensities of $X$, then  they are equivalent relative to $X$. In particular $\FF$-intensity of $X$ is unique up to equivalence relative to $X$.

\item[ii)] Let $\Lambda$  be an $\FF$-intensity of $X$. If $\wh \Lambda$ is an $\FF$-adapted process equivalent to $\Lambda$ relative to $X$, then $\wh \Lambda$  is $\FF$-intensity of $X$.
\end{itemize}
\end{proposition}
\begin{proof}
i)
By assumption, $M$ given by \eqref{eq:Mart-M-nowy} and $\wh M$ defined as
\[
     \wh{M}_t = H_t - \int_0^t \wh \Lambda_u^\top H_{u} du, \quad t\in[0,T],
\]
are $\FF \vee \GG$--$\,$local martingales.
We have that
\[
\wh{M}_t -  M_t = \int_0^t (\Lambda_u - \wh \Lambda_u )^\top H_u du.
\]
Thus $\wh{M} - M$ is a continuous finite variation $\FF \vee \GG$-martingale starting from $0$, and hence it is a constant null process.  Thus \eqref{eq:rown} holds.

ii)
Note that \eqref{eq:rown} implies that for $\FF \vee \GG$ martingale ${M}$ given by \eqref{eq:Mart-M-nowy} it holds
\[
     {M}_t
     =
     H_t - \int_0^t \Lambda_u^\top H_{u} du
     +
     \int_0^t (\Lambda_u - \wh \Lambda_u )^\top H_u du
     =
     H_t - \int_0^t \wh \Lambda_u^\top H_{u} du
     , \quad t\in[0,T].
\]
Thus $\wh \Lambda$ is an $\FF$-intensity of $X$.
\finproof
\end{proof}


In {\cite[Example 3.9]{BieJakNie2014b}} we exhibit an $(\FF,\GG)$-CMC $X$, which admits two different intensities $\Gamma$ and $\Lambda$ that are eqivalent relative to $X$.

In the case of classical Markov chains with finite state space, intensity matrix may not exist if the matrix of transition probabilities is not differentiable (e.g. when $X$ is not quasi left continuous). In the case of $(\FF,\GG)$-CMC the situation is similar. That is, there exist $(\FF,\GG)$-CMCs that do not admit $\FF$-intensities. We illustrate this possibility by means of the following example (see \cite{BieJakNie2015} for details):
\begin{example}
Suppose that   $(\Omega, \mathcal{A}, \mathbb{P})$ supports a real valued standard  Brownian motion $W,$  and a  random variable $E$ with unit exponential distribution\footnote{That is, $E$ is exponentially distributed with mean $1$.} and independent from $W.$
Define a nonnegative process $\gamma$, by formula
\[
	\gamma_t := \sup_{u \in [0,t]} W_u, \quad t\geq 0.
\]
By definition, $\gamma$ is an increasing and continuous process. It is well known (cf. Section 1.7 in  It\^{o} and McKean \cite{ItoMcK1974}) that trajectories of $\gamma$ are not absolutely continuous with respect to the Lebesgue measure on real line. It is shown in \cite{BieJakNie2015} that the process $X$ defined by
\[
	X_t := \I_\set{ \tau \leq  t}, \quad t \geq 0,
\]
where
\[
	\tau := \inf \set{ t > 0 : \gamma_t > E }
\]
is an $(\FF^W, \FF^X)$-CMC which does not admit an $\FF^W$-intensity matrix.
\end{example}
Theorem \ref{thm:int-comp} below provides more insight into the issue of existence of $\FF$-intensity for an $(\FF,\GG)$-CMC.

\subsubsection{Intensity of an $(\FF,\GG)$-CMC and $\FF \vee \GG$-compensators of counting processes $H^{xy}$}

 The  $\FF$-intensity matrix of an $(\FF,\GG)$-CMC $X$ is related to the  $\FF \vee \GG$-compensators of processes $H^{xy}$, $x,y \in S$, $x\neq y$. In fact, we have the following result, which is a special case of    \cite[Lemma 4.3 ]{JakNie2010}, which deals with general jump semimartingales, and thus its proof is omitted.
\begin{theorem}\label{thm:int-comp}
 Let $X$ be an $(\FF,\GG)$-CMC.
\begin{itemize}
\item[ \rm 1)]
Suppose that $X$ admits an $\FF$-intensity matrix process $\Lambda$. Then for every $x,y \in S$, $x\neq y$, the process $H^{xy}$ admits an absolutely continuous $\FF \vee \GG$--compensator given as $\int_0^\cdot H^{x}_u \lambda^{xy}_u du,$ i.e.  the
process $K^{xy}$ defined by
\begin{align}\label{eq:Mxy}
	K^{xy}_t =  H^{xy}_t - \int_0^t H^{x}_u \lambda^{xy}_u du, \quad \tT,
\end{align}
is an $\FF \vee \GG\,$--$\,$local martingale.

\item[\rm 2)]
Suppose that we are given a family of nonnegative $\FF$-progressively measurable  processes $\lambda^{xy}$, $x,y \in S$, $x\neq y$, such that
for every $x,y \in S$, $x\neq y$, the process $K^{xy}$ given in \eqref{eq:Mxy}
is an $\FF \vee \GG\,$--$\,$local martingale. Then, the  matrix valued process $\Lambda_t=[\lambda^{xy}_t]_{x,y \in S }$, with diagonal elements defined as
\[
	\lambda^{xx} =
-\sum_{y \in S, y \neq x} \lambda^{xy}, \quad x\in S,
\]
is an $\FF$-intensity matrix of $X$.
\end{itemize}
\end{theorem}

We see that the $\FF$-intensity may not exist since $\FF \vee \GG$-compensators of $H^{xy}$ may not be absolutely continuous with respect to Lebesgue measure.
 On the other hand, absolute continuity of $\FF \vee \GG$-compensators of all processes
$H^{xy}$, for $x{, y} \in S$, $x\ne y$, is not sufficient for existence of an $\FF$-intensity. This is due to the fact that the density of $\FF \vee \GG$ compensator is, in general,  $\FF \vee \GG$-adapted, whereas  the $\FF$-intensity is only $\FF$-adapted.

In order to focus our study, we now introduce the following restriction:
\begin{center} \framebox[1.02\width][c]{ \strut In the rest of this paper we restrict ourselves to CMCs, which admit $\FF$-intensity.}
\end{center}

\noindent
CMCs that do not admit intensities will be studied in a follow-up paper.

\subsection{$(\FF,\GG)$-CMC as a pure jump semimartingale}

It is important to note that an $(\FF,\GG)$-CMC  $X$ admitting $\FF$-intensity process $\Lambda$  can be viewed as a pure jump  semimartingale,\footnote{We adhere to the standard convention that semimartingale processes (taking values in finite dimensional spaces) are c\` adl\` ag.} with values in $S$, whose corresponding random jump measure $\mu$ defined by (cf. Jacod \cite{Jac1975})
\[
	\mu(\omega, dt,dz )=\sum_{ s < T } \delta_{(s, \Delta X_{s}(\omega)}) (dt,dz) \I_\set{  \Delta X_{s}(\omega)\ne 0} = \sum_{ n \geq 1 } \delta_{(T_n(\omega), {\Delta X_{T_n(\omega)}}(\omega))} (dt,dz) \I_\set{ T_n(\omega) < T},
\]
where
\[
	T_n := \inf \set{ t: T_{n-1}<t\leq T,\ X_t \neq X_{T_{n-1}}}\wedge T, \quad T_0 = 0,
\]
has the $\FF \vee \GG$ predictable projection under $\P$ (the $(\FF \vee \GG, \P)$--compensator)  given as
\begin{align}\label{eq:nu}
	\nu(\omega, dt,dz) = \sum_{x \in S} H^{x}_{t}  \Big( \sum_{y \in S \setminus \set{x}}
	\delta_{y-x}(dz)  \lambda^{xy}_t \Big)dt
	=
\sum_{x \in S} \I_\set{ X_t = x}  \Big( \sum_{y \in S \setminus \set{x}}
	\delta_{y-x}(dz)  \lambda^{xy}_t \Big)dt .
\end{align}
So the problem of construction of  an $(\FF,\GG)$-CMC with an $\FF$-intensity (matrix) process $\Lambda$ is equivalent to the problem of construction of any $\GG$-adapted, $S$-valued pure jump semimartingale with the $(\FF \vee \GG, \P)$--compensator $\nu$ given by \eqref{eq:nu}, and additionally satisfying condition \eqref{eq:CMC-def-1}.

\begin{remark}\label{rem:important}
With a slight abuse of terminology, we shall refer to  a $\GG$-adapted, $S$-valued pure jump semimartingale $X$ with the $\FF \vee \GG$ compensator $\nu$ given by \eqref{eq:nu}, as to a $\GG$-adapted, $S$-valued pure jump semimartingale admitting the  $\FF$-intensity process $\Lambda$. In particular, this also means that the process $M$ corresponding to $X$ as in \eqref{eq:Mart-M-nowy} (see Remark \ref{rem:cacy})  is an $\FF \vee \GG\,$--$\,$local martingale {and,}
even though $X$ is not necessarily  $(\FF, \GG)$-CMC, the conclusions  1) and 2) of Theorem \ref{thm:int-comp} hold.
\end{remark}

Theorem \ref{thm:CMC-char} below shows that a $\GG$-adapted, $S$-valued pure jump semimartingale admitting $\FF$-intensity process $\Lambda$ is, under some additional conditions,  an $(\FF, \GG)$-CMC with the same $\FF$-intensity process $\Lambda$. Before stating the theorem, we recall the notion of immersion between two filtrations.

\bd
We say that a filtration $\FF$ is $\P$-{\it immersed} in a filtration { $\HH$} if $\FF \subset \HH$ and if every $(\P,\FF)\,$--$\,$local martingale is a  $(\P,\HH)\,$--$\,$local martingale.
\ed

\noindent We now have,

\begin{theorem}\label{thm:CMC-char}
Assume that \begin{equation}
\label{asm:immersion}
\ \FF\ \textrm{is $\P$--immersed in}\ \FF \vee \GG.
\end{equation}
Let $X$ be a $\GG$-adapted, $S$-valued pure jump semimartingale admitting the  $\FF$-intensity process $\Lambda$.
Moreover suppose that
\begin{align}
\label{asm:ort-to-M}
 & \ \textrm{all real valued}\  \FF-\textrm{local martingales are orthogonal to components  $M^x, \, x\in S,$} \\
& \textrm{of process $M$  given by \eqref{eq:Mart-M-nowy}}. \nonumber
\end{align}
 Then $X$ is an $(\FF, \GG)$-CMC with the  $\FF$-intensity process $\Lambda.$\footnote{We refer to He, Wang and Yan \cite[Definition 7.33]{HeWanYan1992}, for notion of orthogonality of local martingales.}
\end{theorem}

\begin{proof}
Let us fix $0 = t_0 \leq t_1\leq \ldots \leq t_k\leq T,$ and $x_1,\ldots,x_k\in S.$
It is enough to show that the martingale $N$, given as
\[
N_t = \P( X_{t_1} =x_1, \ldots, X_{t_k} = x_k  | \F_t \vee \G_t), \quad \tT,
\]
is such that $N_t$ is $\F_t \vee \sigma(X_t)$ measurable for any $t \in [0,t_1]$. Indeed, this implies that
\[
	\P(  X_{t_1} =x_1, \ldots, X_{t_k} = x_k |\F_t \vee \G_t)
=
	\P(  X_{t_1} =x_1, \ldots, X_{t_k} = x_k |\F_t \vee \sigma(X_t)),\quad t \in [0,t_1],
\]
which is the $(\FF ,\GG)$-CMC property.
To this end, for each $n=1, \ldots, k,$ we define a process $V^n_t$ by  \[
V^n_t := \prod_{l=1}^{n-1} \I_\set{X_{t_l} = x_l} H^\top_t \E\bigg(Z_t Y_{t_n}e_{x_n} \prod_{m=n}^{k-1} e^{\top}_{x_m}Z_{t_m} Y_{t_{m+1}}e_{x_{m+1}}| \F_t \bigg),\quad \tT,
\]
where $e_{x}$ denotes a column vector in $\r^d$ with $1$ at the coordinate corresponding to state $x$ and with zeros otherwise, and
 $Z$, $Y$ are solutions of the random ODE's\footnote{The symbol $"\mathrm{I}``$  used below is a generic symbol for the identity matrix, whose dimension may vary depending on the context.}
\begin{align*}
d Z_t  &= - \Lambda_t Z_t dt, \quad\quad Z_0 = \mathrm{I}, \quad \tT,
\\
d Y_t  &= Y_t \Lambda_t  dt, \quad\quad\quad Y_0 = \textnormal{I}, \quad \tT.
\end{align*}
 We will show, that
\begin{equation}\label{coin}
V^n_t=N_t, \ \textrm{for}\ t\in  [t_{n-1}, t_n],\ n=1,2,\ldots,k,
\end{equation}
which, in particular, implies that  for every $t \in [0,t_1]$ the random variable $N_t = V^1_t$ is measurable with respect to $\F_t \vee \sigma(X_t)$.

We first note that, in view of Lemma \ref{lem:aux-mart} in Appendix B, the process $V^n$ is
an $\FF \vee \GG$ martingale on $[t_{n-1}, t_n]$.
Moreover,  we have that
\begin{equation}\label{eq:compatib}
V^n_{t_n} = V^{n+1}_{t_n}.
\end{equation}
Indeed,

\begin{align*}
V^{n+1}_{t_n}
&=
\prod_{l=1}^{n} \I_\set{X_{t_l} = x_l} H^\top_{t_n}  \E\bigg(Z_{t_n} Y_{t_{n+1}}e_{x_{n+1}}   \prod_{m=n+1}^{k-1} e^{\top}_{x_m}Z_{t_m} Y_{t_{m+1}}e_{x_{m+1}}| \F_{t_n} \bigg)
\\
&=
\prod_{l=1}^{n-1} \I_\set{X_{t_l} = x_l} H^\top_{t_n} e_{x_n}  H^\top_{t_n} \E\bigg( Z_{t_n} Y_{t_{n+1}}e_{x_{n+1}} \prod_{m=n+1}^{k-1} e^{\top}_{x_m}Z_{t_m} Y_{t_{m+1}}e_{x_{m+1}}| \F_{t_n} \bigg)
\\
&=
\prod_{l=1}^{n-1} \I_\set{X_{t_l} = x_l} H^\top_{t_n}\E\bigg(
Z_{t_n}
Y_{t_n}e_{x_n} e_{x_n}^\top
Z_{t_n} Y_{t_{n+1}}e_{x_{n+1}} \prod_{m=n+1}^{k-1} e^{\top}_{x_m}Z_{t_m} Y_{t_{m+1}}e_{x_{m+1}}| \F_{t_n} \bigg)
\\
&=
\prod_{l=1}^{n-1} \I_\set{X_{t_l} = x_l}  H^\top_{t_n}
\E\bigg(Z_{t_n}
Y_{t_n}e_{x_n}
\prod_{m=n}^{k-1} e^{\top}_{x_m}Z_{t_m} Y_{t_{m+1}}e_{x_{m+1}}| \F_{t_n} \bigg) = V^n_{t_n},
\end{align*}
where the third equality follows from Lemma \ref{eq:YZY-are-bounded} formula \eqref{app1}, and from the fact that
\[
H^\top_{t_n} e_{x_n}  H^\top_{t_n} = H_{t_n}^\top e_{x_n} e_{x_n}^\top.
\]
We will finish the proof by demonstrating \eqref{coin} with use of  backward induction. Towards this end, we start from the last interval, i.e. $n=k$. Observing that
\[
V^k_{t_k}  = \prod_{l=1}^{k-1} \I_\set{X_{t_l} = x_l} H^\top_{t_k} \E\bigg(Z_{t_k} Y_{t_k}e_{x_k} | \F_{t_k} \bigg)
=
\prod_{l=1}^{k-1} \I_\set{X_{t_l} = x_l} H^\top_{t_k} e_{x_k}
=
\prod_{l=1}^{k} \I_\set{X_{t_l} = x_l}
,
\]
and using the martingale property of $V^k$ on  $[t_{k-1},t_{k}],$ we conclude that for $t\in [t_{k-1},t_{k}]$
\[
V^k_t = \E ( V^k_{t_k} | \F_t \vee \G_t ) = \P(X_{t_1} =x_1, \ldots, X_{t_k} = x_k | \F_t \vee \G_t) = N_t .
\]
Now, suppose that for some $n =2, \ldots, k-1,$ the process $V^n$ coincides with $N$ on $[t_{n-1}, t_n]$. This, together with \eqref{eq:compatib}, yields that
\[
N_{t_{n-1}}   = V^n_{t_{n-1}} = V^{n-1}_{t_{n-1}}.
\]
Thus, by the martingale property of $V^{n-1}$ on
 the interval $[t_{n-2},t_{n-1}]$, we obtain
 that
\[
	V^{n-1}_t = \E ( N_{t_{n-1}}   | \F_t \vee \G_t) =  N_t, \quad
t \in [t_{n-2},t_{n-1}].
\]
So the (backward) induction principle completes the proof.

\finproof
\end{proof}

\brem\lab{rem:ort-no-cj}
	A sufficient condition for orthogonality of real valued $\FF\,$--$\,$local martingales and components of process $M$ is that $\FF\,$--$\,$local martingales and the process $M$ do not have common jumps or, equivalently, that $\FF\,$--$\,$local martingales and the process $X$ do not have common jumps.
Indeed, let $Z$ be an $(\FF, \P)\,$--$\,$local martingale. Since $M^x$ is  a local martingale of finite variation we have that
\[
[Z,M^x]_t = \sum_{0 < u \leq t} \Delta Z_u \Delta M^x_u
=
\sum_{0 < u \leq t} \Delta Z_u \Delta H^x_u,
\quad
\tT.
\]
Now, note that $X$ jumps iff one of the processes $H^x, x \in S$, jumps.
Thus if $X$ and $Z$ do not have common jumps then $[Z,M^x]$ is the null process, hence it is a local martingale.
Consequently $Z$ and $M^x$ are orthogonal local martingales.
\erem

We complete this section with the following proposition, which furnishes an interesting example of filtrations $\FF$ and $\GG$ that satisfy conditions \eqref{asm:immersion}  and
\eqref{asm:ort-to-M} of Theorem \ref{thm:CMC-char}.
\begin{proposition}
Let $X$ be an $S$-valued pure jump semimartingale in its own filtration. Let $W$ be a Brownian motion in filtration $\FF^W \vee \FF^X.$
Then \eqref{asm:immersion}  and
\eqref{asm:ort-to-M} of Theorem \ref{thm:CMC-char} are satisfied
with $\GG = \FF^X$ and $\FF = \FF^W$.
\end{proposition}

\begin{proof}
Note that  $W$ is also $\FF^W$  Brownian motion, thus any square integrable $\FF^W$-martingale $N$ can be represented as
\[
	N_t = N_0 + \int_0^t \phi_u d W_u, \quad t \in [0,T],
\]
for some $\FF$-predictable process $\phi$. The assumption that $W$ is a Brownian motion in $\FF^W \vee \FF^X$, implies that $N$ is also an $\FF^W \vee \FF^X$-martingale. This proves  that $\FF^W$ is immersed in $\FF^W \vee \FF^X$.
So \eqref{asm:immersion}  holds. Condition
\eqref{asm:ort-to-M} is satisfied, since all $\FF^W$ martingales are continuous.
\end{proof}

\brem
The assumption that $W$ is a Brownian motion in the filtration $\FF^W \vee \FF^X$ is in fact equivalent to immersion of $\FF^W$ in $\FF^W \vee \FF^X$.
\erem

\section{Construction of $(\FF,\GG)$--CMC via change of measure}\label{constr}

The construction of CMC given in this section generalizes the construction provided in \cite{BieJakNie2015}. In \cite{BieJakNie2015} the authors constructed CMCs that are starting from a given state with probability one. Here, we construct a process $(X_t)_{\tT}$  such that $X$ is an $(\FF, \GG)$--CMC  with the $\FF$-intensity matrix process $\Lambda$, and with $X_0$ satisfying
\be\lab{eq:desired-0}
\P( X_0 = x | \F_T ) = \P( X_0 = x | \F_0 ),\quad x\in S.
\ee
 Even though  in case of ordinary Markov chains a construction of a chain starting from a given state with probability one directly leads to construction of a chain with arbitrary initial distribution, this is not the case any more when one deals with CMCs. In fact, some non-trivial modifications of  the construction argument used in \cite{BieJakNie2015}  will need  to be introduced below.

\subsection{Preliminaries}

 In our construction we start from some underlying   probability space, say $(\Omega,{\mathcal A},\Q)$, on which we are given:
\begin{itemize}
\item[(I1)]$\ $ A (reference) filtration $\FF$.
\item[(I2)]$\ $ An  $S$-valued  random variable $\xi$, such that for any $x\in S$ we have that
\be
 \lab{mux}
	\Q( \xi = x | \F_T ) = \mu_x,
\ee
for some $\F_0\,$--$\,$measurable random variable $\mu_x$ taking values in $[0,1]$.
\item[(I3)]$\ $ A family
$\cN=(N^{xy})_{\substack{x,y \in S \\ y\neq x}}$ of mutually independent Poisson processes, that are independent of $\F_T \vee \sigma(\xi)$ and with non-negative intensities $(a^{xy})_{\substack{x,y \in S \\ y\neq x}}$ (of course for $a^{xy} = 0$ we put  $N^{xy} = 0$).
\end{itemize}

\brem
We observe that condition \eqref{mux} is satisfied iff $\xi=\xi'+\xi'',$ where $\xi'$  is $\F_0\,$--$\,$measurable  and $\xi''$ is orthogonal to $\cF_T$\footnote{A random variable $\xi''$ and sigma field $\cF_T$ are said to be orthogonal if $E_\Q(\xi'' \eta)=0,$ for every $\eta\in L^\infty(\cF_T)$.}. Indeed. First, assume \eqref{mux}. Since conditional expectation with respect to $\mathcal{F}_T$ is an orthogonal  projection from $L^2(\Omega,  \mathcal{A}, \mathbb{Q})$ onto
$L^2(\Omega,\mathcal{F}_T, \mathbb{Q})$ we have that $\xi = \xi' + \xi''$ where $\xi'$ is $\mathcal{F}_T$ measurable and $\xi''$ is orthogonal to $\mathcal{F}_T$, so that $\mathbb{E}(\xi''|\mathcal{F}_T)=0$. {Using \eqref{mux} we see that $\xi'$ is $\mathcal{F}_0$ measurable. }Conversely, if $\xi = \xi' + \xi''$, where  $\xi'$ is $\mathcal{F}_0$ measurable and $\xi''$ is orthogonal to $\mathcal{F}_T$ then \eqref{mux} is obviously satisfied.
\erem

In what follows we take
\be\label{eq:GGhat}
 \G_t= \Big (\bigvee _{\substack{x,y \in S \\ y\neq x}}\F^{N^{xy}}_t\Big ) \vee \sigma(\xi),\ee
which is known to be right-continuous (see \cite[Proposition 3.3]{Ame2000} ).

Given $\GG$ defined in \eqref{eq:GGhat}, we will construct $\wh \GG$-Markov chain, say $X$, as a solution of an appropriate stochastic differential equation. This is an intermediate step in our goal of constructing an $(\FF, \GG)$--CMC  with the $\FF$-intensity matrix process $\Lambda$, and with $X_0$ satisfying
\be\lab{eq:desired-1}
\P( X_0 = x | \F_T ) = \P( X_0 = x | \F_0 ),\quad x\in S,
\ee
for a measure $\P$ to be constructed later.
\bp
Let $A=[a^{xy}]_{x,y \in S},$ where the diagonal elements of $A$ are defined as $a^{xx}:=-\sum_{ \substack{y\in S \\ y\neq x} } a^{xy}.$ Assume that $\xi$ is an $S$-valued random variable  and $\cN = (N^{x,y})$ are Poisson processes satisfying (I3).
Then the unique strong solution of the  following SDE
\begin{equation}\label{eq:SDE}
 d X_t = \sum_{ \substack{x,y\in S \\ x\neq y} } (y - x) \I_\set{x} (X_{t-}) d
 N^{xy}_t,\quad \tT,
\quad X_0 = \xi,
\end{equation}
 is a $\wh{\GG}\,$--$\,$Markov chain
with the infinitesimal generator $A$. Moreover, $A$ is an $\FF$-intensity of $X$ under $\Q$.

\ep
\begin{proof} In view of (I3), the processes $N^{xy}$ and $N^{xy'}$, $y\ne y'$, do not jump together. Thus,
 the process $H^{xy}$ defined for $x,y \in S$, $x\neq y$  by
\[
	H^{xy}_t = \int_0^t H^x_{u-} d
 N^{xy}_u,\quad \tT,
\]
(cf. \eqref{eq:def-Hi} for definition of $ H^x$) counts number of transitions of $X$ from state $x$ to state $y$.  Independence of $N^{xy}$ from $\F_T\vee \sigma(\xi)$ implies that $N^{xy}$ is also a $\wh{\GG}$-Poisson processes with intensity $a^{xy}$. Thus, by boundedness and  $\wh{\GG}$-predictability of $(H^x_{t-})_{ \tT}$, the process $L^{xy}$ given as
\be\lab{comp}
	L^{xy}_t =
\int_0^t H^x_{u-} (d
 N^{xy}_u - a^{xy} du)
=
H^{xy}_t - \int_0^t H^x_{u-}  a^{xy} du=
H^{xy}_t - \int_0^t H^x_{u}  a^{xy} du, \quad t \in [0,T],
\ee
is a $\wh{\GG}$-martingale. Consequently, application of relevant characterization theorem \cite[Thm. 4.1]{JakNie2010} yields that $X$ is a $\wh{\GG}$-Markov chain with the infinitesimal generator $A$.


To finish the proof we observe that since $X$ given by \eqref{eq:SDE} is a pure jump process with finite variation, it is a semimartingale. The $(\wh \GG,\Q)$--compensator of the jump measure of $X$, that is, the jump characteristic of $X$ relative to $(\wh \GG,\Q)$, is given in terms of matrix $A$ (cf. \eqref{comp}). Moreover, since $X$ is adapted to filtration  $\FF\vee \GG\subseteq \wh \GG$, then we see that $X$ is a semimartingale with the $(\FF\vee \GG,\Q)$--compensator of its jump measure  given in terms of matrix $A$. Now,
$A$ is $\FF$--adapted (since it is deterministic), so, in view of the terminology introduced earlier (cf. Definition \ref{def:F-intensity}), $A$ is an $\FF$-intensity of $X$ under $\Q$.
\finproof
\end{proof}
The fact that $X$ is a Markov chain in filtration $\wh{\GG}$ will be critically important below.

\subsection{Canonical conditions}

Let $\Lambda_t=[\lambda^{xy}_t]_{x,y \in S }$ be matrix valued process satisfying the following conditions:
\begin{description}

\item[(C1)] $\Lambda$ is an $\FF$-progressively measurable  and it fulfills \eqref{eq:int-cond}.

\item[(C2)] The processes $\lambda^{xy}$,  $x,y \in S,$ $x\neq y,$ have countably many jumps $\Q$-a.s, and their trajectories admit left limits.
\end{description}
\bd\label{cancon}  A process $\Lambda $ satysfying conditions (C1) and (C2) is said to satisfy canonical conditions relative to the pair $(S,\FF).$
\ed
Any  $\FF$-adapted \cadlag process $\Lambda_t=[\lambda^{xy}_t]_{x,y \in S }$, for which \eqref{eq:int-cond} holds, satisfies canonical conditions.

\medskip\noindent
We are now ready to proceed with construction of a CMC via change of measure.

\subsection{Construction of a CMC}

In this section we provide a construction of a probability measure $\P$, under which the process $X$ {following  the dynamics \eqref{eq:SDE}} is an $(\FF,\GG)$-CMC with a given $\FF$-intensity matrix $\Lambda$ and with $\F_T$-conditional initial distribution satisfying \eqref{eq:desired-1}.

\begin{theorem}\label{thm:CMC-under-constr}
Let $\Lambda$ satisfy canonical conditions relative to the pair $(S,\FF)$ and assume that $\xi$ satisfies (I2). Suppose that $a^{xy}$, introduced in (I3),  is strictly positive for all $x,y \in S$, $x \neq y$.  Moreover, let $X$ be the unique solution of SDE \eqref{eq:SDE}.
For each pair $x,y \in S, x\neq y,$ define the processes $\kappa^{xy}$ as
\[
\kappa^{xy}_t= \frac{{\lambda_{t-}^{xy}}}{a^{xy}} -1,\quad t \in [0,T],
\]
and assume that the random variable $\vartheta$ given as
\[
    \vartheta = \prod_{ x,y \in S : x \neq y }\exp \bigg( - \int_0^T H^x_{u-} a^{xy} \kappa^{xy}_u du \bigg)  \prod_{0 < u \leq T} (1 + \kappa^{xy}_u \Delta H^{xy}_u),
\]
satisfies $\E_\Q \vartheta =1$.\footnote{There exist many different sufficient conditions ensuring  that {$\E_\Q \vartheta = 1$}. For example uniform boundedness of $\Lambda$ is such a condition.} Finally, define on $(\Omega, \wh{\G}_T)$ the probability $\P$ by
\begin{equation}\label{eq:dQ-dP}
\frac{d \P}{d \Q} \big|_{\wh{\G}_T } = \vartheta.
\end{equation}
Then

\begin{itemize}
\item[(i)] \be\lab{obciecie} \P_{|_{\F_T}} = \Q_{|_{\F_T}},\ee
\item[(ii)]  $X$ is an $(\FF, \GG)$--CMC under $\P$ with the $\FF$-intensity matrix process $\Lambda$, and with the initial distribution satisfying
\be\lab{eq:desired}
\P( X_0 = x | \F_T ) = \P( X_0 = x | \F_0 ) = \Q( X_0 = x | \F_T ),\quad x\in S.
\ee
\end{itemize}
\end{theorem}
\begin{proof}
In view of Theorem \ref{thm:CMC-char}, in order to prove (ii) it  suffices to prove that:\\
(a) under measure $\P$ process $X$ has an $\FF$-intensity $\Lambda$, \\
(b) $\FF$ is $\P$--immersed in $\FF \vee \GG $,\\
(c) all real valued $(\FF, \P)$-martingales are orthogonal (under $\P$) to martingales $M^{x}$, $x \in S$,\\
(d) \eqref{eq:desired} holds.

\noindent
We will prove these claims in separate steps. In the process, we will also demonstrate (i).

\noindent
\underline{Step 1:} Here we will show that
$\Lambda$ is an $\FF$-intensity of $X$ under $\P$.
Towards this end, we consider a $\wh{\GG}$-adapted process $\eta$ given as
\[
    \eta_t = \prod_{ x,y \in S : x \neq y }\exp \bigg( - \int_0^t H^x_{u-} a^{xy} \kappa^{xy}_u du \bigg)  \prod_{0 < u \leq t} (1 + \kappa^{xy}_u \Delta H^{xy}_u),\quad \tT,
\]
so that
\[
	d \eta_t = \eta_{t-} \left( \sum_{x,y \in S : x \neq y} \kappa^{xy}_t d L^{xy}_t \right), \quad \eta_0 =1,
\]
where $L^{xy}$ is a $(\wh{\GG}, \Q)$-martingale given by \eqref{comp}. Consequently, process $\eta$ is a $(\wh{\GG}, \Q)$-local martingale. Now, note that $\eta_T=\vartheta$, and thus  $\E_\Q \eta_T=1= \eta_0 $. Thus $\eta$ is  $(\wh{\GG}, \Q)$--martingale (on $[0,T]$).

Since $\kappa^{xy}$ is a left-continuous  and $\FF$-adapted process, and since $\FF \subset \wh{\GG}$,  we conclude that $\kappa^{xy}$ is $\wh{\GG}$-predictable.
Thus, by the Girsanov theorem (see Br{\'e}maud \cite[Thm. VI.T3]{Bre1981}), we conclude that the  $(\wh{\GG}, \P)$ compensator of $H^{xy}$ has density
with respect to the Lebesgue  measure given as\footnote{We use the usual convention that $U_{0-}:=0$ for any real valued process $U$.}
\[
\I_\set{x} (X_{t-})a^{xy}(1 + \kappa^{xy}_t)
=
\I_\set{x} (X_{t-})a^{xy}\left(1 + \frac{ \lambda^{xy}_{t-} }{a^{xy}} -1 \right)
=
\I_\set{x} (X_{t-})\lambda^{xy}_{t-}, \quad t \in [0,T].
\]
So, for any $x\ne y$, the process $\wh K^{xy}$ defined as
\[
	\wh K^{xy}_t :=
H^{xy}_t - \int_0^t \I_\set{x} (X_{u-})  \lambda^{xy}_{u-} du,
\]
is a $\wh{\GG}\,$--$\,$local martingale under $\P$.
Since $X$ is a \cadlag process and since $\lambda^{xy}$ satisfies condition (C2) we see that
\begin{equation}\label{emka}
\wh K^{xy}_t  = H^{xy}_t - \int_0^t H^x_{u} \lambda^{xy}_{u} du, \quad t \in [0,T],
\end{equation}
is a $\wh{\GG}\,$--$\,$local martingale under $\P$.
 Note that  $\FF \vee \GG \subset \wh{\GG}$ and that the process $\wh K^{xy}$ is $\FF \vee \GG$-adapted. Taking 
\[
\tau_n := 
\inf { \Big\{ t \geq 0 : H^{xy}_t \geq n \ \textrm{or} \ \int_0^t   \lambda^{xy}_u  du
\geq  n  \Big\} }
\]
gives us a reducing sequence of $\wh{\GG}$ stopping times for $\wh K^{xy}$, which are also $\FF \vee \GG$ stopping times. So, in view of  \cite[Theorem 3.7]{FolProt2011}, we have that
$\wh K^{xy}$ is also  a $\FF \vee \GG\,$--$\,$local martingale.
Thus according to  Remark \ref{rem:important} we can  use Theorem \ref{thm:int-comp} to conclude that $\Lambda$ is an $\FF$-intensity of $X$ under $\P$.
%
%

\underline{Step 2:} 
In this step we prove \eqref{obciecie}.
By definition of $\P$ and by the tower property of conditional expectations we conclude  that for an arbitrary $\psi \in L^\infty(\F_T)	$ we have
\[
\E_\P(\psi)
=
\E_\Q( \psi \eta_T )
=
\E_\Q( \E_\Q( \psi \eta_T | \wh \G_0 ))
=
\E_\Q( \psi \E_\Q( \eta_T | \wh \G_0 ))
=
\E_\Q( \psi).
\]

\underline{Step 3:}
Next, we show that $\FF$ is $\P$-immersed in $\FF \vee \GG$. In view of
Proposition 5.9.1.1 in Jeanablanc, Yor and Chesney \cite{JeaYorChe2009}{ it suffices} to show that for any $\psi \in L^\infty(\F_T)$ and  any $\tT$ it holds that
\be\lab{rhslhs}
\E_\P ( \psi | \F_t \vee \G_t)
=
\E_\P ( \psi | \F_t )
, \quad
\P-a.s.
\ee
Now, observe that
\[\P(\eta_t>0)=\E_\Q(\1_{\set{\eta_t>0}}\eta_T)\geq \E_\Q(\1_{\set{\eta_T>0}}\eta_T)=\E_\Q(\eta_T)=1,\]
so that $\P( \eta_t > 0 ) =1$. {Moreover, $\eta_t$ is $\F_t \vee  \G_t$ measurable by (I3), \eqref{comp} and (C1).} Thus we have
\[
\E_\P ( \psi | \F_t \vee \G_t)
=
\frac{\E_\Q ( \psi \eta_T |\F_t \vee  \G_t)}
{\E_\Q ( \eta_T | \F_t \vee \G_t)}
=
\frac{\E_\Q ( \psi \E_\Q ( \eta_T | \wh  \G_t)| \F_t \vee \G_t)}
{\eta_t}
\]
\[
=
\frac{\E_\Q ( \psi \eta_t | \F_t \vee \G_t)}
{\eta_t}
=
\E_\Q ( \psi | \F_t \vee \G_t)	
=
\E_\Q ( \psi | \F_t ),  \quad
\P-a.s.,	\]
where the third equality holds in view of the fact that $\eta$ is $(\wh{\GG},\Q)$-martingale, and where the last equality holds since $\FF$ is $\Q$-immersed in $\FF \vee \GG$ (see Appendix A, Corollary  \ref{cor:immersion}).
Hence, using \eqref{obciecie} we conclude
\[
\E_\P ( \psi | \F_t \vee \G_t)
=
\E_\P ( \psi | \F_t )
,
\quad
\P-a.s.
\]

Consequently, \eqref{rhslhs} holds.

\underline{Step 4:}
Now we show the required orthogonality, that is we prove claim (c). Towards this end it suffices to prove that all real valued $(\FF,\P)$-martingales do not have common jumps with $X$ under $\P$ (see Remark \ref{rem:ort-no-cj}).
Let us take $Z$ to be an arbitrary real valued $(\FF, \P)$-martingale. Then, in view of \eqref{obciecie}
$Z$ is an $(\FF, \Q)$-martingale. By  (I3),  we have that $(\FF,\Q)$-martingales and Poisson processes in $\cN$ are independent under $\Q$. Thus, by Lemma  \ref{lem:indep-Poisson} in the Appendix A, the $\Q$ probability that process $Z$ has common jumps with any process from family $\cN$ is zero. Consequently, in view of \eqref{eq:SDE},  the $(\FF,\Q)$-martingale $Z$ does not jump together with $X$, $\Q$-a.s.
Therefore, by absolute continuity of $\P$ with respect to $\Q$, $\P$ probability that $Z$ jumps at the same time as $X$ is zero.

%

\underline{Step 5:} Finally, we will show that \eqref{eq:desired} holds.
Towards this end, let us take an arbitrary real valued function $h$ on $S$. The abstract Bayes rule yields
\begin{align*}
	\E_\P( h(X_0) | \F_T  )
&=
\frac{ \E_\Q( h(X_0) \eta_T | \F_T  ) }{ \E_\Q( \eta_T | \F_T  ) }
=
\frac{ \E_\Q( \E_\Q(  h(X_0) \eta_T | \wh \G_0) | \F_T  ) }{ \E_\Q( \E_\Q(  \eta_T |\wh \G_0) | \F_T  ) }
\\
&
=
\E_\Q( h(X_0) \E_\Q(  \eta_T | \wh \G_0) | \F_T  )
=\E_\Q( h(X_0) | \F_T  )
=\E_\Q( h(X_0) | \F_0  ),
\end{align*}
where the last equality follows from the fact that by assumption
\eqref{mux} the initial condition of the process $X$ satisfies
\be\lab{rzucik}
\Q( X_0 = x | \F_T ) = \Q( X_0 = x | \F_0 ),\quad x\in S.
\ee
Consequently,
\begin{align*}
\E_\P( 	h(X_0) | \F_0) &=
\E_\P( 	\E_\P( h(X_0) | \F_T  ) | \F_0) =
\E_\P( 	\E_\Q( h(X_0) | \F_0  ) | \F_0) =
\E_\Q( h(X_0) | \F_0  ) \\
&= 	\E_\P( h(X_0) | \F_T  ).
\end{align*}
This completes the proof of \eqref{eq:desired}, and the proof of the theorem.
\finproof
\end{proof}

\section{$(\FF, \GG)$--CMC vs $(\FF, \GG)$--DSMC}\label{cmcdsmc}

In this section we first re-visit the concept of the  doubly stochastic Markov chain. Then, we 
study relationships between conditional Markov chains and doubly stochastic Markov chains. These
relationships are crucial for the theory of consistency of CMCs and for the theory of CMC copulae, that are put forth in the  companion paper \cite{BieJakNie2014b}.

\subsection{ $(\FF, \GG)$--DSMC}

We start with introducing the concept of  $(\mathbb{F},\mathbb{G})$-doubly stochastic Markov chain ($(\FF, \GG)$--DSMC for brevity), which generalizes the notion of  $\mathbb{F}$-doubly stochastic Markov chain (cf. \cite{JakNie2010}), as well as  the notion of continuous time $\mathbb{G}$-Markov chain.

\begin{definition}\label{def:DSMC}
A  $\mathbb{G}$-adapted c\` adl\` ag  process $X=(X_t)_{t \in [0,T]}$ is called  an $(\mathbb{F},\mathbb{G})$--doubly
stochastic Markov chain with state space $S$  if
for any $0 \leq s \leq t \leq T $ and for
    every $y \in S$ it holds that
 \be\lab{eq:DSMC-def}
 \mathbb{P}( X_t = y \mid \mathcal{F}_T \vee \mathcal{G}_s )
 =    \mathbb{P}( X_t = y \mid \mathcal{F}_t \vee \sigma(X_s) ).
 \ee
\end{definition}

 We refer to  \cite{JakNie2010} for examples of processes, which are $(\FF, \FF^X)$-DSMCs. We remark that in \cite{JakNie2010}  it was assumed that the chain $X$ starts from some point $x\in S$ with probability one, whereas here, we allow for the initial state $X_0$ to  be a non-constant random variable.

With any $X$, which is an $(\FF,\GG)$-DSMC,  we associate a matrix valued random field $ P=( P(s,t),\  0 \leq s \leq t \leq T)$, where $ P(s,t)= ( p_{xy}(s,t))_{x,y \in
S }$   is defined by
\be\label{tp}   p_{x,y}(s,t) =
    \frac{ \mathbb{P}( X_t = y , X_s = x \mid \mathcal{F}_t  ) }{  \mathbb{P}(  X_s = x  | \F_t ) }
\1_\set{\mathbb{P}(  X_s = x  | \F_t ) > 0}
+
\I_\set{ x=y} \1_\set{\mathbb{P}(  X_s = x  | \F_t ) = 0}
.
\ee
The following result provides a characterization of $(\FF,\GG)$-DSMC. \begin{proposition}\label{prop:c-tf}
A process $X$ is an $(\FF,\GG)$-DSMC iff there exists a stochastic matrix valued random field $\wt P(s,t) = ( \wt p_{xy}(s,t))_{x,y \in
S },\ 0 \leq s \leq t \leq T$, such that:
\begin{itemize}
\setlength\itemsep{-0em}
\item[1)] for  every $s\in [0,T]$, the process $\wt P(s,\cdot)$ is $\FF$--adapted on $[s,T]$,
\item[2)]  for any $0 \leq s \leq t \leq T $ and for
    every $x,y \in S$ we have
    \begin{equation}\label{eq:DSMC-def-2}
    \I_\set{ X_s  = x }\mathbb{P}( X_t = y \mid \mathcal{F}_T \vee \mathcal{G}_s )
     = \I_\set{ X_s  = x }
    \wt p_{xy}(s,t)
    \end{equation}
or, equivalently,
for
    every $y \in S$ we have
    \begin{equation}\label{eq:DSMC-def-3}
    \mathbb{P}( X_t = y \mid \mathcal{F}_T \vee \mathcal{G}_s )
     = \sum_{x \in S}\I_\set{ X_s  = x }
    \wt p_{xy}(s,t).
    \end{equation}
\end{itemize}
\end{proposition}
\begin{proof}  We first prove the sufficiency.
 Using \eqref{eq:DSMC-def-2} we have that
\be\lab{ten}
  \mathbb{P}( X_t = y \mid \mathcal{F}_T \vee \mathcal{G}_s )
 =
 \sum_{x \in S} \I_\set{X_s = x}  \wt p_{xy}(s,t).
\ee
So, taking conditional expectations with respect to $\mathcal{F}_t \vee \sigma(X_s)$ on both sides of \eqref{ten}, observing that $\mathcal{F}_t \vee \sigma(X_s) \subset \mathcal{F}_T \vee \mathcal{G}_s$, and  using the tower property of conditional expectations,  we obtain
 \[
  \mathbb{P}( X_t = y \mid \mathcal{F}_t \vee \sigma(X_s) )
 =
 \mathbb{E}\bigg(
 \sum_{x \in S} \I_\set{X_s = x}  \wt{p}_{xy}(s,t)  \mid \mathcal{F}_t \vee \sigma(X_s)\bigg)
 =
 \sum_{x \in S} \I_\set{X_s = x}  \wt{p}_{xy}(s,t),
\]
where the last equality follows from measurability of  $\sum_{x \in S} \I_\set{X_s = x}  \wt{p}_{xy}(s,t)$ with respect to $\mathcal{F}_t \vee \sigma(X_s)$.
This and  \eqref{ten} imply
 \[
 \mathbb{P}( X_t = y \mid \mathcal{F}_T \vee \mathcal{G}_s )
 =
 \mathbb{P}( X_t = y \mid \mathcal{F}_t \vee \sigma(X_s) ),
 \]
which is \eqref{eq:DSMC-def}.

 Now we prove the necessity.
First we observe that, using similar arguments as in Jakubowski and Niew\k{e}g\l owski  \cite[Lemma 3]{JakNie2007} (see also Bielecki, Cr\'{e}pey, Jeanblanc and Rutkowski  \cite[Lemma 2.1]{BielCrepJeanRutk2006}), we have that  for $t \geq s$
\begin{align}\label{JNBCJR}
&\mathbb{P}( X_t = y \mid \mathcal{F}_t \vee \sigma(X_s) ) \\
\nonumber
&= \sum_{x \in S }
\I_\set{ X_s  = x }
\left(
\frac{ \mathbb{P}( X_t = y , X_s = x \mid \mathcal{F}_t  ) }{  \mathbb{P}(  X_s = x  | \F_t ) } \1_\set{ \mathbb{P}(  X_s = x  | \F_t ) > 0 }
+ \1_\set{y=x} \I_\set{ \P(X_s = x  | \F_t ) = 0 }
\right) \quad
\mathbb{P}-a.s.
\end{align}
Consequently, in view of \eqref{tp} we have
\[
\mathbb{P}( X_t = y \mid \mathcal{F}_t \vee \sigma(X_s) )
=
\sum_{x \in S }
\I_\set{ X_s  = x } p_{x,y}(s,t).
\]
It is enough now to let $\wt p_{x,y}(s,t)= p_{x,y}(s,t),$ for $x,y \in
S , 0 \leq s \leq t \leq T$.
\finproof
\end{proof}

As we saw in the proof of Proposition \ref{prop:c-tf} we can take $\wt P = P$, where $P$ is given by  \eqref{tp}.
\brem Observe that, in view of  the results in Rao \cite{Rao1972}, we have that for every $s\in [0,T]$ and almost every $\omega \in \Omega$   the function $P(s,\cdot)$ is measurable on $[s,T]$, and that for every $t\in [0,T]$ and almost every $\omega \in \Omega$, the function $P(\cdot,t)$ is measurable on $[0,t]$.
\erem

This, and \eqref{eq:DSMC-def-2} justify the following definition
\begin{definition}
The matrix valued random field $ P=( P(s,t),\  0 \leq s \leq t \leq T)$,
defined by \eqref{tp} is called the conditional transition
probability matrix field (c--transition field for short) of $X$.
\end{definition}

\brem
For the future reference, we note  that \eqref{eq:DSMC-def-3}  in the definition of
an $(\mathbb{F},\mathbb{G})$-DSMC, can be written in the following form (recall that we take $\wt P = P$):
    \[
\mathbb{E} ( H^y_t \mid \mathcal{F}_T \vee \mathcal{G}_s) =
\sum_{x \in S} H^x_s p_{xy}(s,t)\qquad
\text{
for
    every $y \in S$,}
\]
which is equivalent to
\begin{equation}\label{eq:DSMC-vector}
\mathbb{E} \left( H_t \mid \mathcal{F}_T  \vee  \mathcal{G}_s
\right) = P(s,t)^\top H_s .
\end{equation}
\erem
%

We know that in the case of classical Markov chains the transition semigroup and the initial distribution of the chain characterize the finite dimensional distributions of the chain, and thus they characterize the  law of the chain. The next proposition shows that, in case of an $(\mathbb{F},\mathbb{G})$--DSMC $X$,  the c-transition field $P$ of $X$ and the conditional law of $X_0$ given $\mathcal{F}_T$ characterize conditional law of $X$ given $\mathcal{F}_T$.
\begin{proposition} \label{wn2.4}
If $X$ is an $(\mathbb{F},\mathbb{G})$--DSMC with c-transition field $P,$ then for arbitrary $0 =t_0 \leq t_1 \leq \ldots \leq t_n \leq t  \leq T$ and $(x_1, \ldots,
x_n ) \in S^n$ it holds that
\begin{equation}\label{eq:c-fidis}
\mathbb{P}( X_{t_1} = x_1, \ldots X_{t_n} = x_n \mid
\mathcal{F}_T ) =\sum_{x_0 \in S} \mathbb{P}( X_{0} = x_0 | \F_T)
\prod_{k=0}^{n-1}
p_{x_{k} , x_{k+1}} (t_{k}, t_{k+1}).
\end{equation}
Moreover, if
\begin{equation}\label{eq:c-init-dist}
\mathbb{P}( X_{0} = x_0 | \F_T) =\mathbb{P}( X_{0} = x_0 | \F_0) \quad
\text{for every $x_0 \in S$,}
\end{equation}
then for arbitrary $0 \leq t_1 \leq \ldots \leq t_n \leq t  \leq T$ and $(x_1, \ldots,
x_n ) \in S^n$
\begin{equation}\label{eq:cond-ind}
\begin{aligned}[c]
 \mathbb{P}( X_{t_1} = x_1, \ldots X_{t_n} = x_n \mid
\mathcal{F}_T ) = \mathbb{P}( X_{t_1} = x_1, \ldots X_{t_n} =
x_n \mid \mathcal{F}_t ).
\end{aligned}\end{equation}
\end{proposition}
\begin{proof}
Let us fix arbitrary $x_1, \ldots, x_n \in S$ and  $0  \leq t_1 \leq \ldots \leq t_n \leq t  \leq T$, and let us define the set $A$ by
\[
A = \set{ X_{t_1} =x_1, \ldots, X_{t_n} = x_n  }.
\]
Note that by Lemma 3.1  in \cite{JakNie2010}  we have
\begin{equation*}
\P( A | \F_T \vee \G_0)
\I_\set{X_0 = x_0}
=\I_\set{ X_{0} = x_0 }
\prod_{k=0}^{n-1} p_{x_{k} , x_{k+1}} (t_{k}, t_{k+1}).
\end{equation*}
Consequently,
\begin{align*}
\mathbb{P}( A \mid
\mathcal{F}_T ) &=\sum_{x_0 \in S} \mathbb{P}( X_{0} = x_0 | \F_T)
\prod_{k=0}^{n-1}
p_{x_{k} , x_{k+1}} (t_{k}, t_{k+1})
\end{align*}
which proves \eqref{eq:c-fidis}. Thus, in view of \eqref{eq:c-init-dist}, the following equality is satisfied
\begin{align*}
\P(A| \mathcal{F}_T)
&=\sum_{x_0 \in S} \mathbb{P}( X_{0} = x_0 | \F_0)  \prod_{k=0}^{n-1}
p_{x_{k} , x_{k+1}} (t_{k}, t_{k+1}).
\end{align*}
Since $P$ is a c-transition field we obtain that $\P(A| \mathcal{F}_T)$ is $\mathcal{F}_t$ measurable as a product of $\mathcal{F}_t$-measurable random variables. Thus, the tower property of conditional expectations yields
\eqref{eq:cond-ind}.
\finproof
\end{proof}
\begin{corollary}\label{cor:hihihaha}
Let $X$ be an $(\FF, \GG)$-DSMC with $X_0$ satisfying \eqref{eq:c-init-dist}.
Then $\FF$ is  $\P$-immersed in $\FF \vee \FF^X$.
\end{corollary}
\begin{proof}
In view of Proposition \ref{wn2.4}  process $X$ satisfies \eqref{eq:cond-ind}. This, by  \cite[Lemma 2]{JakNie2007}, is  equivalent to $\P$-immersion of $\mathbb{F}$ in $\mathbb{F} \vee \mathbb{F}^X$.
\finproof
\end{proof}
In analogy to the concept of $\FF$-intensity for $(\mathbb{F},\mathbb{G})$-CMCs, one considers the concept of intensity with regard to $(\mathbb{F},\mathbb{G})$-DSMCs.
 Definition \ref{defdef} introduces a concept of such intensity. This definition is stated in the form, which is consistent with the way the original definition of intensity for DSMCs {was} introduced in \cite{JakNie2010}.
Later on, we will show that this definition can be equivalently stated in the form analogous to Definition \ref{def:F-intensity}.
\begin{definition}\label{defdef}
We say that  {an} $\mathbb{F}$--adapted matrix-valued
process $\Gamma = (\Gamma_s)_{s \geq 0} = ([\gamma^{xy}_s]_{x,y \in S})_{s
\geq 0}$ is an intensity of  $(\mathbb{F},\mathbb{G})$-DSMC $X$
if the following conditions are satisfied: \\
1)
\begin{equation}  \label{nr2/8}
\int_{]0,T]} \sum_{x \in S} \abs{\gamma^{xx}_s } ds <
\infty  .\ \ \ \ \ \
\end{equation}
 2)
  \begin{equation}\label{eq:intensity-cond}
   \gamma^{xy}_s \geq 0 \ \ \
    \forall x, y\in S, x\neq y
    ,
    \ \ \ \
    \gamma^{xx}_s = - \sum_{y \in S: y \neq x} \gamma^{xy}_s
    \ \ \ \
    \forall x \in  S.
\end{equation}
3)
  The Kolmogorov backward equation holds: for all  $v\leq t$,
\begin{equation}\label{eq:INT-trans-prob-backward}
 P(v,t) - \mathrm{I}     =   \int_v^t  \Gamma_u P(u,t) du.
\end{equation}
4)  The Kolmogorov forward equation holds: for all  $v\leq t$,
\begin{equation}\label{eq:INT-trans-prob-forward}
 P(v,t)  - \mathrm{I} = \int_v^t P(v,u) \Gamma_u du.
\end{equation}
\end{definition}
\begin{remark}\label{rem:Kol-Eq}
The above Kolmogorov equations admit unique solution provided that $\Gamma$ satisfies \eqref{nr2/8}.
 The unique solution of Kolmogorov equation
\eqref{eq:INT-trans-prob-backward} is given by the formula which is known as Peano-Baker series
\begin{equation*}
P(v,t) = \mathrm{I} + \sum_{n=1}^\infty \int_{v}^{t}
\int_{v_1}^{t} \ldots \int_{v_{n-1}}^{t} \Gamma_{v_1}\ldots
\Gamma_{v_n} d v_n \ldots d v_1,
\end{equation*}
and the solution of \eqref{eq:INT-trans-prob-forward} is given by
\[
P(v,t) = \mathrm{I} + \sum_{n=1}^\infty \int^{t}_{v}
\int^{v_1}_{v} \ldots \int^{v_{n-1}} _{v} \Gamma_{v_n}\ldots
\Gamma_{v_1} d v_n \ldots d v_1.
\]
There is also a different useful representation of  the solution of Kolmogorov equations.  It is given in terms of a matrix exponential, and it is called the Magnus expansion:
\[
P(v,t) = \exp(\Phi(v,t)),
\]
where $\Phi(v,t)$ is the Magnus series
\[
\Phi(v,t) = \sum_{k=1}^\infty \Phi_k(v,t).
\]
We refer to Blanes, Casas, Oteo and Ros  \cite{BlaCasOteRos2009} for a detailed statement of the Magnus expansion in deterministic case. The formulae found in \cite{BlaCasOteRos2009} are adequate in our case, as here we use the Magnus expansion of $P(v,t)$ for every $\omega \in \Omega$.

It is easily seen from the Magnus expansion, that $P(v,t)$ has inverse $Q(v,t)=\exp(-\Phi(v,t))$. {For an alternative proof of invertibility of $P(v,t)$ we refer to
\cite[Proposition 3.11.iii)]{JakNie2010}}.

\end{remark}

\subsubsection{Martingale characterizations of $(\mathbb{F},\mathbb{G})$-DSMC}

It turns out that  the $(\mathbb{F},\mathbb{G})$-DSMC property of process $X$ is fully characterized by the martingale property (\wrt  the filtration $\mathbb{\widehat{G}}$ given by \eqref{eq:GGhat-0}) of
some processes related to $X$. These characterizations are given in the next theorem.

\begin{theorem}\label{thm:characterization-DSMC}
Let $(X_t)_{t \in [0,T] }$ be an $S$--valued stochastic
process and $(\Gamma_t)_{t \in [0,T] }$ be an $\FF$-adapted matrix valued process
satisfying \eqref{nr2/8} and \eqref{eq:intensity-cond}. The following conditions are equivalent:
\\
i) The process $X$ is an $(\mathbb{F},\mathbb{G})$-DSMC with the
intensity process $\Gamma$.
\\
ii) The processes $\wh{M}^x$ defined by
\begin{equation}  \label{nr1/8}
\wh{M}^x_t := H^{x}_t - \int_{]0,t]} \gamma^{X_u,x}_u du , \ \quad x \in S,
\end{equation}
 are $\mathbb{\widehat{G}}\,$--$\,$local martingales.
\\
iii) The processes $K^{xy}$ defined by
\begin{equation}  \label{nr1/9}
    K^{xy}_t := H^{xy}_t - \int_{]0,t]} H^x_s \gamma^{xy}_s
    ds, \quad x,y \in S, \ x \neq y ,
\end{equation}
  where
\begin{equation}\label{eq:def-Hij}
    H^{xy}_t := \int_{]0,t]} H^x_{u-} d H^y_u,
\end{equation}
are $\mathbb{\widehat{G}}\,$--$\,$local martingales.
\\
iv) The process  $L$ defined by
\begin{equation}  \label{eq:def-l}
    L_t := Z^\top_t H_t,
\end{equation}
where $Z$ is a unique solution to the random integral equation
\begin{equation}\label{eq:def-Q0}
    d Z_t = - \Gamma_t Z_t dt, \ \ \ Z_0 = \mathrm{I} ,
\end{equation}
is a $\mathbb{\widehat{G}}\,$--$\,$local martingale.
\\
v) For any $t \in [0,T]$, the process $N^t$ defined as
\begin{equation}\label{eq:N-mart}
N^t_s := P(s,t)^\top H_s \ \ \ \ \text{ for } 0 \leq s \leq t .
\end{equation}
is  a $\widehat{\mathbb{G}}$ martingale,
{ where
\[
P(s,t) :=  Z_s Y_t
\]
with
\begin{equation*}
d Y_t  = Y_t \Gamma_t  dt, \ \ \ \ \ Y_0 = \mathrm{I}, \quad \tT,
\end{equation*}
}
\end{theorem}
\begin{proof}
The proof of equivalence of (i)--(iv) goes along the lines of the proof of \cite[Theorem 4.1]{JakNie2010}; only minor and straightforward modifications are needed,  and therefore the proof is omitted.
{
Equivalence of (iv) and (v) follows from formula
\[
	N^{t}_s = Y_t^\top L_s 
 \ \ \ \ \text{ for } 0 \leq s \leq t
\]
and the fact that $Y_t$ is uniformly bounded $\wh{\cG}_0$ measurable invertible matrix (Lemma  \ref{eq:YZY-are-bounded}).}
\finproof
 \end{proof}


The following result is direct counterpart of Proposition \ref{rem:int-wersje} and therefore we omit its proof.
\begin{proposition}\label{rem:int-wersje-2}
Let $X$ be an $(\FF,\GG)$-DSMC.
\begin{itemize}
\item[i)] If $\Gamma$ and $\wh \Gamma$ are intensities of $X$, then they are equivalent relative to $X$. In particular intensity of $X$ is unique up to equivalence relative to $X$.

\item[ii)] Let $\Gamma$  be an intensity of $X$. If $\wh \Gamma$ is an $\FF$-adapted process equivalent to $\Gamma$ relative to $X$, then $\wh \Gamma$  is intensity of $X$.
\end{itemize}
\end{proposition}

We will not discuss here the question of existence of an $(\FF, \GG)$-DSMC with intensity $(\Gamma_t)_{t \in [0,T] }$. This question will be addressed in some generality in Bielecki, Jakubowski and Niew\k{e}g{\l}owski \cite{BJN-buczek}. Instead, in the next section, we will show that any $(\FF, \GG)$-CMC process $X$ constructed in Theorem \ref{thm:CMC-under-constr} is also an $(\FF, \GG)$-DSMC.

Since an $(\FF,\GG)$-DSMC $X$ is a $S$-valued \cadlag process, then  it is a pure jump semimartingale.
This observation sheds a new light on the intensity of $X$ as the following corollary shows.

\begin{corollary}\label{cor:akuku}
Intensity of an $(\FF,\GG)$-DSMC $X$ is an $\FF$-intensity of $X$ in the sense of Definition \ref{def:F-intensity}.
\end{corollary}
\begin{proof}
The process $\wh{M}$ is a $\wh{\GG}$-local martingale by Theorem \ref{thm:characterization-DSMC}.ii).
{ In fact, it is also an $\FF \vee \GG$-local martingale. To see this, we take a reducing sequence of $\wh{\GG}$-stopping times 
\[ \tau_n := \inf \set{ t \geq 0: \int_0^t \sum_{y \in S} |\gamma^{yy}_s| ds \geq n }.
\] Since $\wh{M}$ is  also $\FF \vee \GG$-adapted and $(\tau_n)_{ n \geq 1}$ are also $\FF \vee \GG$ stopping times we see that $\wh{M}$ is an $\FF \vee \GG$-local martingale (see e.g. \cite[Theorem 3.7]{FolProt2011}).}
This implies that the $\FF$-adapted process $\Gamma$ is an $\FF$-intensity of $X$.
%
%
\end{proof}

\subsection{Relation between CMC and DSMC}

In this section we present some aspects of  relationship between the classes of $(\FF,\GG)$-CMCs and $(\FF,\GG)$-DSMCs.

\subsubsection{DSMCs that are CMCs}

\begin{proposition}\label{prop:DSMCisCMC}
Assume that $\FF$ and $\GG$ satisfy the immersion property \eqref{asm:immersion}, and that $X$ is an $(\FF,\GG)$-DSMC. Then $X$ is an $(\FF,\GG)$-CMC. In addition if $X$ considered as an $(\FF,\GG)$-DSMC admits intensity $\Gamma$, then $X$ considered as an $(\FF,\GG)$-CMC admits $\FF$-intensity $\Lambda = \Gamma$.
\end{proposition}
\begin{proof}
Let us fix arbitrary $x_1, \ldots, x_k \in S$ and  $0\leq t \leq t_1 \leq\ldots \leq t_k\leq T,$ and let us define a set $A$ by
\[
A = \set{ X_{t_1} =x_1, \ldots, X_{t_k} = x_k  }.
\]
We need  to show that
\begin{equation}\label{eq:CMC-A1}
\P( A | \F_t \vee \G_t)
=
\P( A | \F_t \vee \sigma(X_t)).
\end{equation}
Towards this end we first note that by Lemma 3.1  in \cite{JakNie2010}  we have
\begin{equation}\label{eq:DSMC-A-LHS}
\P( A | \F_T \vee \G_t)
\I_\set{X_t = x}
=\I_\set{ X_{t} = x } p_{x,x_1}(t,t_1)
\prod_{n=1}^{k-1} p_{x_{n} , x_{n+1}} (t_{n}, t_{n+1}).
\end{equation}
The tower property of conditional expectation and \eqref{eq:DSMC-A-LHS}
imply
\begin{align*}
\P( A | \F_t \vee \G_t)
&=
\E \left( \sum_{x \in S }\E \left(
\I_A| \F_T \vee \G_t \right)
\I_\set{X_t = x}
| \F_t \vee \G_t\right)
\\
&=
\E \left( \sum_{x \in S }
\I_\set{ X_{t} = x } p_{x,x_1}(t,t_1)
\prod_{n=1}^{k-1} p_{x_{n} , x_{n+1}} (t_{n}, t_{n+1})
| \F_t \vee \G_t\right)
\\
&=
\sum_{x \in S }
\I_\set{ X_{t} = x } \E \left( p_{x,x_1}(t,t_1)
\prod_{n=1}^{k-1} p_{x_{n} , x_{n+1}} (t_{n}, t_{n+1})
| \F_t \vee \G_t\right)
.
\end{align*}
Thus using the assumed immersion property  of $\FF$ in $\FF \vee \GG$  we obtain
\begin{align*}
\P( A | \F_t \vee \G_t)&=
\sum_{x \in S }
\I_\set{ X_{t} = x } \E \left( p_{x,x_1}(t,t_1)
\prod_{n=1}^{k-1} p_{x_{n} , x_{n+1}} (t_{n}, t_{n+1})
| \F_t \right),
\end{align*}
{which implies the CMC property.}

The second claim of the theorem follows immediately from Corollary \ref{cor:akuku}.
\finproof \end{proof}
The following example illustrates the use of Proposition \ref{prop:DSMCisCMC}.
\begin{example}(Time changed discrete Markov chain)
 Consider process $\bar{C}$, which is a discrete time Markov chain
with values in $S =\set{1, \ldots ,K }$ and with transition probability matrix $P$. In addition consider process $N$, which is a Cox process with c\` adl\` ag
$\FF$-intensity process $\tilde{\lambda}$.
From \cite[Theorem 7 and
9]{JakNie2007} we know that under assumption that
 the processes $(\bar{C}_k)_{k \geq 0}$ and $(N_t)_{t\in[0,T]}
$ are independent and  conditionally independent given
$\mathcal{F}_T$,  the  process
\[
C_t := \bar{C}_{N_t}
\] is
an $(\mathbb{F}, \mathbb{F}^C)$-DSMC. Moreover $C$ admits intensity process $\Gamma=[\gamma^{xy}]$ given as
 \[
    \gamma^{xy}_t = (P-I)_{x,y} \tilde{\lambda}_t .
    \]
Thus, by Corollary \ref{cor:hihihaha} and Proposition \ref{prop:DSMCisCMC}, the process $C$ is an $(\FF,\FF^C)$-CMC with $\FF$-intensity $\Lambda=\Gamma.$	

\end{example}


\subsubsection{CMCs that are DSMCs}
\begin{theorem}\label{thm:CMC-DSMC-int}
Suppose that $X$ is an $(\FF,\GG)$-CMC admitting an $\FF$-intensity $\Lambda$. In addition, suppose that $X$ is also an $(\FF,\GG)$-DSMC with an intensity $\Gamma$.
Then $\Gamma$ is an $\FF$-intensity of $X$ and $\Lambda$ is an intensity of $X$.
\end{theorem}
\begin{proof}
It follows from Corollary \ref{cor:akuku} that $\Gamma$ is an $\FF$-intensity.
Thus by Proposition \ref{rem:int-wersje} $\Lambda$ and $\Gamma$ are equivalent relative to $X$.
Consequently\j{,} by Proposition \ref{rem:int-wersje-2} process $\Lambda$ is an intensity of $X$.
\finproof
\end{proof}
This and Proposition \ref{wn2.4} imply
\begin{corollary}
If $X$ is an $(\FF,\GG)$-CMC with $\FF$-intensity and also an $(\FF,\GG)$-DSMC with intensity, then $\FF$-intensity (or, equivalently, intensity) and $\F_T$-conditional distribution of $X_0$ determine the $\F_T$-conditional distribution of $X$.
\end{corollary}

In case of process $X$ constructed in Theorem \ref{thm:CMC-under-constr} the result of  Theorem \ref{thm:CMC-DSMC-int} can be strengthen  as follows.
\bp\label{thm:CMC-DCMC}
Let $X$ be a process constructed in Theorem \ref{thm:CMC-under-constr}, so that  $X$ is an  $(\FF, \GG)$-CMC process with an $\FF$-intensity process $\Lambda$.
Then $X$ is also an $(\FF, \GG)$-DSMC with an intensity process $\Gamma = \Lambda$.
\ep
\begin{proof}
In Step 1 of the proof of Theorem \ref{thm:CMC-under-constr} we showed that the processes $\wh{K}^{xy}$, $x,y \in S$, $x \neq y $, given by \eqref{emka}, are $\wh{\GG}\,$--$\,$local martingales. Thus, by Theorem \ref{thm:characterization-DSMC}, $X$ is an $(\FF,\GG)$-DSMC with intensity  $\Lambda$.
\finproof
\end{proof}

%
\subsubsection{Pure jump semimartingales that are both CMCs and DSMCs}

\bt\label{prop:CMCisDSMC}
Let $\FF$, $\GG$ satisfy the immersion property \eqref{asm:immersion}.
Assume that  $S$--valued $\GG$--adapted pure jump semimartingale $X$ admits an $\FF$-intensity $\Lambda$.
Moreover suppose that the orthogonality  property \eqref{asm:ort-to-M} is fulfilled.
 Then $X$ is an $(\FF,\GG)$-CMC and an $(\FF, \GG)$-DSMC with intensity $\Lambda$.
\et
\begin{proof}
In Theorem \ref{thm:CMC-char} we showed that $X$ is an $(\FF,\GG)$-CMC.  In order to prove that $X$ is an $(\FF,\GG)$-DSMC it suffices to show that for
for every
$A \in \F_T$, $B \in \G_t$, $t  \leq u$ and $y \in S$ it holds that
\begin{align}\label{eq:DSMC-inaczej}
\E( \I_A \I_B \I_\set{X_u = y})
=
\E( \I_A \I_B H^\top_t Z_t Y_u e_y),
\end{align}
where $Y$ and  $Z$ are defined by \eqref{eq:random-ODE-Y-pom} and \eqref{eq:random-ODE-X-pom}, respectively.
Indeed, by the monotone class theorem, the above yields
\begin{equation}\label{eq:Nu-mart}
\P({X_u = y} | \F_T \vee \G_t)
=H^\top_t Z_t Y_u e_y.
\end{equation}
Consequently, since the right hand side of \eqref{eq:Nu-mart} is measurable with respect to $\F_{{T}} \vee \sigma(X_t)$ we obtain the desired $(\FF, \GG)$-DSMC property of
$X$.

It remains to  prove \eqref{eq:DSMC-inaczej}.
 {
Since $Z_t Y_{u}e_{y} \I_A \in L^1(\F_T)$ (see Lemma \ref{eq:YZY-are-bounded}), the following formula
\[
V_s :=  \I_B H^\top_s \E\big(Z_s Y_{u}e_{y} \I_A| \F_s \big), \quad s\in [t,u],
\]
well defines a  process $V$ on $[t,u]$. Now, let  a Doob martingale $D$ be defined on $[0,T]$ by
\[
D_s =\E(\I_A \I_B \I_\set{X_u = y} | \F_s \vee \G_s), \quad s \in [0,T].
\]
The immersion property \eqref{asm:immersion} leads  to }%
%
%
%

\[
	V_t = \I_B H^\top_t \E\big(Z_t Y_{u}e_{y} \I_A| \F_t \vee \G_t \big)
=
\E\big(\I_A\I_B H^\top_t Z_t Y_{u}e_{y} | \F_t \vee \G_t \big).
\]
Next, we will show that $V_t =D_t$, which in turn will imply that
\begin{align}\label{eq:DSMC-inaczej-2}
\E(\I_A \I_B \I_\set{X_u = y} ) = \E D_t	= \E V_t = \E (\I_A\I_B H^\top_t Z_t Y_{u}e_{y}),
\end{align}
which is  \eqref{eq:DSMC-inaczej}.

\noindent In order to show that $D_t = V_t$ we will demonstrate a stronger result, namely that $V=D$ on $[t,u]$.
To this end, note that by Lemma \ref{lem:aux-mart} $V$ is {an} $\FF \vee \GG$-martingale on the interval $[t,u]$. Thus, to show that  $V=D$ on $[t,u]$ it suffices to show that $V_u= D_u$. For this purpose, let us define on  $[u,T]$ the process $W$ by
\[
W_s := \I_B \I_\set{X_u = y} \E(\I_A  | \F_s ).
\]
Next we observe that for $s \in [u,T]$
\[
	D_s =\I_B \I_\set{X_u = y} \E(\I_A  | \F_s \vee \G_s)
=
\I_B \I_\set{X_u = y} \E(\I_A  | \F_s ) = W_s,
\]
where the penultimate equality follows from immersion of $\FF$ in $\FF \vee \GG$.
Hence, using the fact that $Z_u Y_u = \mathrm{I} $ (see Lemma \ref{eq:YZY-are-bounded}), we have
\[
D_u = W_u = \I_B \I_\set{X_u = y} \E(\I_A  | \F_u ) = \I_B H^\top_u e_y \E(\I_A  | \F_u )= \I_B H^\top_u \E(Z_u Y_u e_y \I_A  | \F_u )= V_u.
\]
Thus, by the martingale property of $D$ and $V$ (on $[t,u]$), we conclude that $D = V$ on $[t,u]$.
This completes the proof of \eqref{eq:DSMC-inaczej} and, consequently, demonstrates that $X$ is an $(\FF, \GG)$-DSMC.

In order to verify that $X$ admits intensity $\Lambda$ we first note that the random field $P$ defined as
\be\lab{eq:Ptu}
P(t,u) := Z_t Y_u
\ee
solves the Kolmogorov equations \eqref{eq:INT-trans-prob-backward} and  \eqref{eq:INT-trans-prob-forward}. Next we observe that \eqref{eq:Nu-mart} implies the martingale property of $N^u$ given as in \eqref{eq:N-mart}, with $P(t,u)$ as in \eqref{eq:Ptu}.
Thus, by Theorem \ref{thm:characterization-DSMC}, $\Lambda$ is an intensity of $X$.
 The proof of the theorem is now complete.
\finproof\end{proof}

%
\section{Appendices}
\subsection*{Appendix A}
In this appendix we provide technical results needed for derivations done in Section \ref{constr}.
\begin{lemma}\label{pilambda}
Let  $\xi$ be an $S$--valued random variable defined on a filtered probability space $(\Omega, \cA, \HH  ,\wt \P)$ with $\HH =\set{\cH_t}_{t \in [0,T]}$.
Suppose that
\be\lab{eq:cond-FTF0}
 \E_{\wt \P}( h(\xi) | \H_T )
=
 \E_{\wt \P}( h(\xi) | \H_0 )
\ee
for every real valued function $h$ on $S$. Then $\HH$ is ${\wt \P}$--immersed in $\HH \vee \sigma(\xi)$.
\end{lemma}
\begin{proof}
It is sufficient to prove (c.f. \cite[Lemma 6.1.1]{BieRut2004}) that for every $\psi \in L^\infty(\H_T)$ it holds that
\be\lab{eq:immersFxi1}
	\E_{\wt \P} ( \psi | \H_t \vee \sigma(\xi))
=
	\E_{\wt \P} ( \psi | \H_t ), \quad  \forall \tT.
\ee
Let us fix $\tT$ and $\psi \in L^\infty(\H_T)$.
By {the} standard $\pi-\lambda$ system arguments it is enough to show that
\be\lab{eq:immersFxi2}
		\E_{\wt \P} ( \psi \I_A \I_B (\xi) )
=
		\E_{\wt \P} ( \E_{\wt \P} ( \psi | \H_t ) \I_A \I_B(\xi) ), \quad
\forall
A \in \H_t, B \subseteq S,
\ee
where
\[
    \I_B(\xi) =
    \left\{
      \begin{array}{ll}
        1, & \xi \in B,  \\
        0, & \xi \notin B.
      \end{array}
    \right.
\]
Towards this end we first derive another representation of the right hand side in \eqref{eq:immersFxi2},
\begin{align*}
\E_{\wt \P} ( \E_{\wt \P} ( \psi | \H_t ) \I_A \I_B(\xi) )
&=		\E_{\wt \P} ( \E_{\wt \P} ( \psi \I_A | \H_t ) \I_B(\xi) )
=
		\E_{\wt \P} ( \E_{\wt \P} ( \E_{\wt \P} ( \psi \I_A | \H_t ) \I_B(\xi) | \H_T ))
\\
&=
		\E_{\wt \P} ( \E_{\wt \P} ( \psi \I_A | \H_t ) \E_{\wt \P} (  \I_B(\xi) | \H_T ))
=
		\E_{\wt \P} ( \E_{\wt \P} ( \psi \I_A | \H_t ) \E_{\wt \P} (  \I_B(\xi) | \H_0 ))
\\
&=
		\E_{\wt \P} ( \E_{\wt \P} ( \psi \I_A \E_{\wt \P} (  \I_B(\xi) | \H_0 ) | \H_t ) )
=
		\E_{\wt \P} ( \psi \I_A \E_{\wt \P} (  \I_B(\xi) | \H_0 )),
\end{align*}
where the fourth equality follows from \eqref{eq:cond-FTF0}.
The left hand side of \eqref{eq:immersFxi2} can be rewritten as
\begin{align*}
		\E_{\wt \P} ( \psi \I_A \I_B (\xi) )
&=
		\E_{\wt \P} ( \E_{\wt \P} ( \psi \I_A \I_B (\xi) | \H_T))
=
		\E_{\wt \P} ( \psi \I_A \E_{\wt \P} ( \I_B (\xi) | \H_T)) \\
&
=
\E_{\wt \P} ( \psi \I_A \E_{\wt \P} ( \I_B (\xi) | \H_0)),
\end{align*}
where the last equality follows from \eqref{eq:cond-FTF0}.
This proves \eqref{eq:immersFxi2} and thus concludes the proof of the lemma.
\finproof
\end{proof}
\begin{corollary}\lab{cor:immersion}
Let $\bK$ be a filtration on  $(\Omega, \cA  ,\wt \P)$, such that it is independent of  $\HH \vee \sigma(\xi)$.
Suppose that $\xi$ satisfies \eqref{eq:cond-FTF0}. Then
$\HH$ is $\wt \P$--immersed in $ \HH \vee \bK \vee \sigma(\xi)$.
\end{corollary}
\begin{proof}
The result follows from Lemma \ref{pilambda} and from the fact that if $\HH^1$ and $\HH^2$ are two independent filtrations on  $(\Omega, \cA  ,\wt \P)$, then $\HH^1$ is $\wt \P$-immersed in $\HH^1 \vee \HH^2$.
\finproof
\end{proof}
In the next lemma we use the same probabilistic setup as in Section \ref{sec:CMC-and-int}.
\begin{lemma}\lab{lem:indep-Poisson}
Let $X$ be an $\FF$ adapted c\`adl\`ag process, and let $N$ be a Poisson process.
Suppose that $N$ and $\FF$  are independent.
Then
\[
\P\left( \set{ \omega \in \Omega  : \exists t \in [0,T] \ s.t. \  \Delta X_t(\omega) \Delta N_t(\omega) \neq 0 }\right) = 0.
\]
\end{lemma}
\begin{proof}
First note that both  $X$ and $N$ have countable number of jumps on $[0,T]$, and let denote their jump times as $(T_n)_{n\geq 1}$ and  $(S_n)_{n \geq 1}$, respectively.
Independence of $N$ and $\FF$  implies  that $(T_n)_{n\geq 1}$ and $(S_n)_{n\geq 1}$ are independent. Since each random variable  $S_n$ is Gamma distributed  and thus has density,  then for any $n,k \geq 1$ it holds that
\[
\P( T_n= S_k ) =0.
\]
Since
\[
A := \set{ \omega : \exists \tT \ s.t. \  \Delta X_t(\omega)  \Delta N_t(\omega) \neq 0}
=
\bigcup_{n,k \geq 1}
\set{ \omega : T_n(\omega) = S_k(\omega) }
\]
 we have
\[
\P(A)
\leq
\sum_{n,k \geq 1}
\P( T_n= S_k )
=0.
\]
\finproof
\end{proof}

\subsection*{Appendix B}

In this appendix we derive some technical results that are used in Section  \ref{sec:CMC-and-int} and Section \ref{cmcdsmc}.

\begin{lemma}\label{eq:YZY-are-bounded}
Let $Z$ and $Y$ be solutions of the random ODE's
\begin{equation}\label{eq:random-ODE-X-pom}
d Z_t  = - \Psi_t Z_t dt, \ \ \ \ \ Z_0 = \mathrm{I}, \quad \tT,
\end{equation}
\begin{equation}\label{eq:random-ODE-Y-pom}
d Y_t  = Y_t \Psi_t  dt, \ \ \ \ \ Y_0 = \mathrm{I}, \quad \tT,
\end{equation}
where $\Psi$ is an appropriately measurable matrix valued process satisfying \eqref{eq:int-cond}\footnote{For any $\omega$  for which $\Psi$
does not satisfy \eqref{eq:int-cond}, we set $\Psi_t(\omega)=0$ for all $\tT$. } and such that
\be\label{eq:int-integrable}
	\sum_{x \in S } \int_0^T |\psi^{xx}_u| du < \infty.
\ee
Then, the matrix valued random processes   $(Y_t)_{0\leq t \leq T}$ and $(Z_t Y_v)_{0\leq t\leq v}$, $v\in [0,T]$, have elements that are nonnegative  and bounded by $1$. Moreover
\begin{equation}\label{app1}
Z_t Y_t= \mathrm{I}\quad \textrm{for}\  \tT.
\end{equation}
\end{lemma}
\begin{proof}
Using Remark \ref{rem:Kol-Eq} one can verify that for each $t$, the functions $Y_t(\cdot)$ and $Z_t(\cdot)$ are measurable, so that $Y$ and $Z$ are  matrix valued random processes.

Since $\Psi$ satisfies \eqref{eq:int-cond}, then for every $\omega$,  $Y_{\cdot}(\omega)$ is a solution of matrix forward Kolmogorov equation, and so its elements belong to the interval $[0,1]$ (since they give conditional probabilities, see e.g. Gill and Johansen \cite[Thm. 12 and  Thm. 13]{GilJoh1990}).

Next, observe that, letting $Z(t,v) = Z_t Y_v$ we have that
\[
d_t Z(t,v) = (d Z_t) Y_v =  - \Psi_t Z_t Y_v dt =  - \Psi_t Z(t,v)  dt, \quad 0\leq t \leq v.
\]
Moreover, it is easy to verify that $Z(v,v) = Z_v Y_v= Z_0 Y_0=\mathrm{I}$. We thus see that for every $\omega$, $Z(\cdot,v)(\omega)$ satisfies the Kolmogorov backward equation,
\[
d_t Z(t,v) =   - \Psi_t Z(t,v)  dt, \quad 0\leq t \leq v, \quad Z(v,v)=\mathrm{I},
\]
and so, it has non-negative elements bounded by $1$.
\finproof
\end{proof}

The following lemma is used in the proof of Theorem \ref{thm:CMC-char}.
\begin{lemma}\label{lem:aux-mart}
Suppose that assumptions of Theorem \ref{thm:CMC-char} are satisfied.
Let U  be an $\bR^d$-valued bounded random variable, and let $Z$ and $Y$ be solutions of the random ODE's
\eqref{eq:random-ODE-X-pom} and
\eqref{eq:random-ODE-Y-pom}, respectively.
Fix $u$ and $v$ satisfying $0\leq  u < v \leq T,$ and fix  set  $A \in \F_u \vee \G_u$.
Then, process $V$ given by
\[
	V_t = \I_A H^\top_t Z_t\E( Y_{v} U | \F_t), \quad t \in [0,T],
\]
is an $\FF \vee \GG $ martingale on the interval
$[u, v]$.
%
\end{lemma}
\begin{proof}
It suffices to prove that the process $\widehat V$ given as
\[
	\widehat{V}_t  = H^\top_t Z_t  \E(Y_{v} U | \F_t),\quad t \in [0,T],
\]
 is an $\FF \vee \GG$ martingale  on $[0,v]$.
Furthermore, since all components of $H_t$ and $Z_t Y_{v} $ are non-negative and bounded by $1$ (for the latter see Lemma \ref{eq:YZY-are-bounded}), and since random variable $U$ is bounded, then it suffices to show that $\widehat{V}$ is an $\FF \vee \GG$ local martingale.

Towards this end we first verify that vector valued process $L = (L^x,\, x \in S)^\top$ defined by  $L_t := H^\top_t Z_t,\ \tT,$ is an $\FF \vee \GG\,$--$\,$local martingale with the following representation
\begin{align}
L_t = H^\top_0 + \int_0^t d M_u^\top \! \cdot \! Z_u, \quad \tT.
\end{align}
Indeed, since $\Lambda$ is an $\FF$-intensity, integration by parts yields that
\[
d L_t =	d ( H^\top_t Z_t)  = H^\top_{t-} d Z_t + d H^\top_t\! \cdot \!  Z_t
= -H^\top_{t-} \Lambda_t Z_t dt + d H^\top_t\! \cdot \!  Z_t
= d M_t^\top \! \cdot \! Z_t.
\]
Next, we observe that the  vector valued process $U(\cdot,v)= (U^x(\cdot,v),\, x \in S)\mn{^\top}$ defined by
\[
{
U^x(t,v) = \sum_{y \in S }\mathbb{E} \left( Y^{xy}_v U^y | \F_t \right), \quad t \in [0,T],}
\]
is an $\FF$-martingale. Since we assume that $\FF$ is right-continuous we can take right-continuous modification of $U(\cdot,v)$.

Thus, by assumptions \eqref{asm:immersion} and  \eqref{asm:ort-to-M} in Theorem \ref{thm:CMC-char}, its components are orthogonal to components of $M$.
Hence the square bracket processes  $[M^y, U^x(\cdot, v)]$,  $x,y \in S$, are $\FF \vee \GG$-local martingales.
By properties of square brackets (cf. Protter \cite[Thm. II.6.29]{prot2004}) we obtain
\[
	[L^x , U^x(\cdot, v)]_t = \sum_{y \in S} \int_0^t Z^{y,x}_u d [M^y, U^x(\cdot, v)]_u.
\]
Thus, by predictability and  local boundedness of $Z$,  and by \cite[Thm. IV.2.29]{prot2004}, we conclude that process $[L^x , U^x(\cdot, v)]$ is a local martingale,  and consequently that local martingales $L^x$ and $ U^x(\cdot, v)$ are orthogonal. Since,
\[
\widehat{V}_t
=L_t U(t,v)=
\sum_{x \in  S } L^x_t U^x(t,v), \quad t \in [0,T],
\]
we conclude that $\widehat{V}$ is an $\FF \vee \GG\,$--$\,$local martingale  as  a sum of local martingales. \finproof
\end{proof}

\subsection*{Acknowledgments}
Research of T.R. Bielecki was partially supported by NSF grants DMS-0908099 and DMS-1211256.

\input{CMC-construction-11-2-2015-ARXIV.bbl}
\end{document}

%% file: headingsIgor.tex
\usepackage{amsthm, amsmath, amssymb, amsfonts, graphicx, epsfig, bbm}

\usepackage{mathrsfs}

\usepackage{enumerate}

\usepackage[colorlinks=true, pdfstartview=FitV, linkcolor=blue,
            citecolor=blue, urlcolor=blue]{hyperref}
\usepackage[usenames]{color}

\usepackage[numbers]{natbib}

\usepackage{url}
\makeatletter\def\url@leostyle{%
 \@ifundefined{selectfont}{\def\UrlFont{\sf}}{\def\UrlFont{\scriptsize\ttfamily}}} \makeatother

\usepackage{accents, comment}

\usepackage[normalem]{ulem}

\newtheorem{theorem}{Theorem}
\newtheorem{proposition}[theorem]{Proposition}
\newtheorem{lemma}[theorem]{Lemma}
\newtheorem{corollary}[theorem]{Corollary}
\newtheorem{definition}[theorem]{Definition}

\theoremstyle{definition}

\newtheorem{example}[theorem]{Example}
\theoremstyle{remark}
\newtheorem{remark}[theorem]{Remark}

\numberwithin{equation}{section}
\numberwithin{theorem}{section}

\def\cA{\mathcal{A}}

\def\cF{\mathcal{F}}
\def\cG{\mathcal{G}}
\def\cH{\mathcal{H}}

\def\cN{\mathcal{N}}

\def\bK{\mathbb{K}}

\def\bR{\mathbb{R}}

\def\bse{\begin{equation}\begin{split}}
\def\ese{\end{split}\end{equation}}

\newcommand{\1}{\mathbbm{1}}                     
\newcommand{\set}[1]{{\{#1\}}}            
\renewcommand{\mid}{\;|\;}              
\newcommand{\abs}[1]{\left\vert#1\right\vert}   

%% file: CMC-construction-11-2-2015-ARXIV.bbl
\begin{thebibliography}{10}

\bibitem{Ame2000}
J.~Amendinger.
\newblock Martingale representation theorems for initially enlarged
  filtrations.
\newblock {\em Stochastic Process. Appl.}, 89(1):101--116, 2000.

\bibitem{BiaWid2013}
F.~Biagini, A.~Groll, and J.~Widenmann.
\newblock Intensity-based premium evaluation for unemployment insurance
  products.
\newblock {\em Insurance Math. Econom.}, 53(1):302--316, 2013.

\bibitem{BielCrepJeanRutk2006}
T.~R. Bielecki, S.~Cr\'{e}pey, M.~Jeanblanc, and M.~Rutkowski.
\newblock Valuation of basket credit derivatives in the credit migrations
  environment.
\newblock { Birge, J. and Linetsky, V. (ed.) Handbooks in Operations Research
  and Management Science: Financial Engineering, vol.15. 471-510}, 2008.

\bibitem{BJN-buczek}
T.~R. Bielecki, J.~Jakubowski, and M.~Niew\k{e}g{\l}owski.
\newblock Fundamentals of {T}heory of {S}tructured {D}ependence between
  {S}tochastic {P}rocesses.
\newblock Book in preparation.

\bibitem{BieJakNie2015}
T.~R. {Bielecki}, J.~{Jakubowski}, and M.~{Niew\k{e}g{\l}owski}.
\newblock {Conditional Markov chains -- construction and properties.}
\newblock In {\em Banach Center Publications 105(2015) Stochastic analysis.
  Special volume in honour of Jerzy Zabczyk.}, pages 33--42. Warsaw: Polish
  Academy of Sciences, Institute of Mathematics, 2015.

\bibitem{BieJakNie2014b}
T.~R. Bielecki, J.~Jakubowski, and M.~Niew\k{e}g{\l}owski.
\newblock Conditional {M}arkov chains, part {II}: {c}onsistency and {c}opulae.
\newblock
  \href{http://arxiv.org/abs/1501.05535}{(http://arxiv.org/abs/1501.05535)},
  2015.

\bibitem{BieRut2000}
T.~R. Bielecki and M.~Rutkowski.
\newblock Multiple ratings model of defaultable term structure.
\newblock {\em Math. Finance}, 10(2):125--139, 2000.
\newblock INFORMS Applied Probability Conference (Ulm, 1999).

\bibitem{BieRut2003}
T.~R. Bielecki and M.~Rutkowski.
\newblock Dependent defaults and credit migrations.
\newblock {\em Appl. Math. (Warsaw)}, 30(2):121--145, 2003.

\bibitem{BieRut2004}
T.~R. Bielecki and M.~Rutkowski.
\newblock Modeling of the defaultable term structure: {C}onditionally {M}arkov
  approach.
\newblock {\em IEEE Transactions on Automatic Control}, 49:361--373, 2004.

\bibitem{BlaCasOteRos2009}
S.~Blanes, F.~Casas, J.~A. Oteo, and J.~Ros.
\newblock The {M}agnus expansion and some of its applications.
\newblock {\em Phys. Rep.}, 470(5-6):151--238, 2009.

\bibitem{Bre1981}
P.~Br{\'e}maud.
\newblock {\em {P}oint {P}rocesses and {Q}ueues}.
\newblock Springer-Verlag, New York, 1981.
\newblock Martingale dynamics, Springer Series in Statistics.

\bibitem{EbeGrb2013}
E.~Eberlein and Z.~Grbac.
\newblock Rating based {L}\'evy {L}ibor model.
\newblock {\em Math. Finance}, 23(4):591--626, 2013.

\bibitem{EbeOzk2003}
E.~Eberlein and F.~{\"O}zkan.
\newblock The defaultable {L}\'evy term structure: ratings and restructuring.
\newblock {\em Math. Finance}, 13(2):277--300, 2003.

\bibitem{FolProt2011}
H.~F{\"o}llmer and P.~Protter.
\newblock Local martingales and filtration shrinkage.
\newblock {\em ESAIM Probab. Stat.}, 15(In honor of Marc Yor, suppl.):S25--S38,
  2011.

\bibitem{GilJoh1990}
R.~D. Gill and S.~Johansen.
\newblock A survey of product-integration with a view toward application in
  survival analysis.
\newblock {\em Ann. Statist.}, 18(4):1501--1555, 1990.

\bibitem{HeWanYan1992}
S.~W. He, J.~G. Wang, and J.~A. Yan.
\newblock {\em Semimartingale {T}heory and {S}tochastic {C}alculus}.
\newblock Kexue Chubanshe (Science Press), Beijing, 1992.

\bibitem{ItoMcK1974}
K.~It{\^o} and H.~P. McKean, Jr.
\newblock {\em Diffusion {P}rocesses and {T}heir {S}ample {P}aths}.
\newblock Springer-Verlag, Berlin-New York, 1974.
\newblock Second printing, corrected, Die Grundlehren der mathematischen
  Wissenschaften, Band 125.

\bibitem{Jac1975}
J.~Jacod.
\newblock Multivariate point processes: predictable projection,
  {R}adon-{N}ikod\'ym derivatives, representation of martingales.
\newblock {\em Z. Wahrscheinlichkeitstheorie und Verw. Gebiete}, 31:235--253,
  1974/75.

\bibitem{JakNie2007}
J.~Jakubowski and M.~Niew\k{e}g\l{}owski.
\newblock {Pricing bonds and CDS in the model with rating migration induced by
  a Cox process.}
\newblock {Stettner, \L ukasz (ed.), Advances in mathematics of finance. Banach
  Center Publications 83, 159-182.}, 2008.

\bibitem{JakNie2010}
J.~Jakubowski and M.~Niew\k{e}g{\l}owski.
\newblock A class of {$\Bbb F$}-doubly stochastic {M}arkov chains.
\newblock {\em Electron. J. Probab.}, 15:no. 56, 1743--1771, 2010.

\bibitem{JakNie2011}
J.~Jakubowski and M.~Niew\k{e}g{\l}owski.
\newblock Pricing and hedging of rating-sensitive claims modeled by
  {$\mathbb{F}$}-doubly stochastic {M}arkov chains.
\newblock In {\em Advanced mathematical methods for finance}, pages 417--453.
  Springer, Heidelberg, 2011.

\bibitem{JeaYorChe2009}
M.~Jeanblanc, M.~Yor, and M.~Chesney.
\newblock {\em Mathematical methods for financial markets}.
\newblock Springer Finance. Springer-Verlag London, Ltd., London, 2009.

\bibitem{prot2004}
P.~E. Protter.
\newblock {\em Stochastic {I}ntegration and {D}ifferential {E}quations},
  volume~21 of {\em Stochastic Modelling and Applied Probability}.
\newblock Springer-Verlag, Berlin, 2005.
\newblock Second edition. Version 2.1, Corrected third printing.

\bibitem{Rao1972}
M.~Rao.
\newblock On modification theorems.
\newblock {\em Trans. Amer. Math. Soc.}, 167:443--450, 1972.

\end{thebibliography}
